\documentclass[12pt]{amsart}
\usepackage{amsmath,amsfonts,amssymb,amsxtra,amscd,enumerate,amsthm}
\usepackage{color}
\usepackage{latexsym}
\usepackage{epsfig}
  \usepackage{fullpage}

\newcommand\les{\lesssim}
\newcommand\ges{\gtrsim}

\newcommand\R{\mathbb{R}}
\newcommand\T{\mathbb{T}}

\newcommand\Z{\mathbb{Z}}
\newcommand\N{\mathbb{N}}
\renewcommand\S{\mathbb{S}}

\newcommand{\calQ}{\mathcal Q}
\newcommand{\calT}{\mathcal T}

\renewcommand{\L}{\mathcal{L}}

\newcommand{\ls}{{\lesssim}}

\newcommand\la{\langle}
\newcommand\ra{\rangle}
\newtheorem{theo}{Theorem}
\numberwithin{theo}{section} 
\newtheorem{lema}[theo]{Lemma}
\newtheorem{prop}[theo]{Proposition}
\newtheorem{corol}[theo]{Corollary}
\newtheorem{defin}[theo]{Definition}

\newtheorem{rem}{Remark}

\newtheorem{con}{Condition}

\numberwithin{equation}{section}

\begin{document}
\title{Optimal bilinear restriction estimates for general hypersurfaces and the role of the shape operator}

\author[I. Bejenaru]{Ioan Bejenaru} \address{Department
  of Mathematics, University of California, San Diego, La Jolla, CA
  92093-0112 USA} \email{ibejenaru@math.ucsd.edu}

\begin{abstract} It is known that under some transversality and curvature assumptions on the hypersurfaces involved, the bilinear restriction estimate holds true with better exponents than what would trivially follow from the corresponding linear estimates. This subject was extensively studied for conic and parabolic surfaces with sharp results proved by Wolff \cite{Wo} and Tao \cite{Tao-BW, Tao-BP}, and with later generalizations \cite{Lee-BR,LeeVa, BuMuVa, Sto}. In this paper we provide a unified theory for general hypersurfaces and clarify the role of curvature in this problem, by making statements in terms of the shape operators of the hypersurfaces involved.

\end{abstract}


\maketitle

\section{Introduction}

A fundamental question in Harmonic Analysis is the restriction estimate, to which we will refer as the linear restriction estimate. Given a smooth compact hypersurface $S \subset \R^{n+1}, n \geq 1$ the linear restriction estimate $R_S(p \rightarrow q)$ holds true if 
\begin{equation} \label{LRE} 
\| \hat f \|_{L^q(S,d\sigma)} \leq C(p,S) \| f \|_{L^p(\R^{n+1})}.
\end{equation}
We recall that a hypersurface in $\R^{n+1}$ is an $n$-dimensional submanifold of $\R^{n+1}$. 
\eqref{LRE} justifies why, given $f \in L^p(\R^{n+1})$, one can meaningfully consider the restriction of $\hat f$ to $S$ as an element $L^q(S,d\sigma)$. A priori 
such a result is not to be expected, unless $p=1$ (indeed if $S$ is a subset of a hyperplane the above estimate will fail for $p > 1$). However it is well known
that if $S$ has some non-vanishing principal curvatures, then improvements are available beyond the trivial case $p=1$. 

Using duality, the linear restriction estimate $R_S(p \rightarrow q)$ is equivalent to the adjoint linear restriction estimate $R_S^{*}(q' \rightarrow p')$: 
\begin{equation} \label{ALRE} 
\| \widehat{f d \sigma} \|_{L^{p'}(\R^{n+1})} \leq C(p,S) \| f \|_{L^{q'}(S,d\sigma)}.
\end{equation}
where $p',q'$ are the dual exponents to $p,q$ used in \eqref{LRE}. Establishing  \eqref{LRE} or  \eqref{ALRE} for the full conjectured range of pairs $(p,q)$, respectively $(p',q')$ is a major open problem in Harmonic Analysis. However in the case $q=q'=2$, the problem is completely 
understood with optimal ranges for $p,p'$. We note that the linear restriction estimate and its adjoint are known to hold true for other values of $q,q'$, that is $q \ne 2$, respectively $q' \ne 2$. 

Throughout the rest of this paper we set $q=q'=2$.  In this case, the optimal range for $p'$ in \eqref{ALRE} is well-known: if $S$ has non-zero Gaussian curvature then $R_S^{*}(2 \rightarrow \frac{2(n+2)}{n})$ holds true, 
 if $S$ has one vanishing principal curvature and the others are non-zero, then $R_S^{*}(2 \rightarrow \frac{2(n+1)}{n-1})$ holds true, 
and so on (keeping in mind that it is necessary that at least one principal curvature has to be non-zero). The original formulation of the linear restriction estimate was made for surfaces with non-zero Gaussian curvature 
and the result is due to Tomas-Stein, see \cite{St}. In Partial Differential Equations the adjoint linear restriction estimate is known as 
a particular case of the more general class Strichartz estimates $L^p_t L^q_x$; the adjoint restriction estimate occurs when $p=q$, see 
\cite{Tao-book}. 

For reasons that we explain later, it is important to consider the bilinear adjoint restriction estimate 
$R_{S_1,S_2}^{*}(2 \times 2 \rightarrow p)$:
\begin{equation} \label{ABRE}
\| \widehat{f_1 d \sigma_1} \cdot \widehat{f_2 d\sigma_2} \|_{L^{p}(\R^{n+1})} \leq C(p,S_1,S_2) \| f_1 \|_{L^2(S_1,d\sigma_1)}
\| f_2 \|_{L^2(S_2,d\sigma_2)}.
\end{equation}
We abuse language throughout the rest of this paper and will refer to \eqref{ABRE} as the bilinear restriction estimate, thus skipping the adjoint part.  This is also consistent with title of our paper. 

If $S_1=S_2=S$ then $R_S^{*}(2 \times 2 \rightarrow p)$ is equivalent to $R_S^{*}(2 \rightarrow 2p)$.
However if $S_1$ and $S_2$ have some transversality properties, it is expected that \eqref{ABRE}
improves the range of $p$ over what follows directly from its linear counterpart, that is if $p_{linear}$ is the optimal exponent 
in $R_S^{*}(2 \rightarrow p)$, then $R_S^{*}(2 \times 2 \rightarrow p)$ may hold true for $p < \frac{p_{linear}}2$. 

Klainerman and Machedon conjectured that \eqref{ABRE} holds true for $p \geq p_0= \frac{n+3}{n+1}$ in the case of conic 
(model $\tau=|\xi|$) and quadratic surfaces (model $\tau=|\xi|^2$ or $|(\tau,\xi)|=1$).  A formalization of this conjecture was made in \cite{FoKl} by Foschi and Klainerman. Essentially this states that if $S_i=\{ (\tau,\varphi(\xi)); \xi \in D_i \}$ (where $\varphi(\xi)$ is one of the models) and $D_1$ and $D_2$ are "separated" then \eqref{ABRE} holds true. By constructing counterexamples, Foschi and Klainerman show that the exponent $p_0= \frac{n+3}{n+1}$ is optimal, in the sense that \eqref{ABRE} cannot hold true if $p < p_0$ for either of the models listed. 

This conjecture is expected to hold true for more general hypersurfaces (than the models listed above). One of the main 
goals of this paper is to understand what are the natural conditions that $S_1$ and $S_2$ need to satisfy such that \eqref{ABRE}
holds true for optimal $p \geq p_0= \frac{n+3}{n+1}$. But first, we review the current state of the conjecture. 

In reading the results below it is important to keep in mind that \eqref{ABRE} is trivial for $p=\infty$, and, with a little more work,
it can be shown to be true for $p=2$ using only transversality hypothesis (we recall here that $S_1$ and $S_2$ are smooth and compact). In fact, in the absence of any curvature assumptions, it can be shown that $p=2$ is optimal by simply taking $S_1,S_2$ to be compact subsets of transversal hyperplanes. The difficult part of \eqref{ABRE} is using curvature information in order to obtain results with $p <2$.  

For conic surfaces, the estimate \eqref{ABRE} was formulated by Bourgain. In \cite{Bou-CM} Bourgain proved it for $n=2$ and $p=2-\epsilon$, and in \cite{TV-CM1} Tao and Vargas proved it for $n=2$ and $p > 2-\frac{8}{121}$, as well as 
$n=3$ and $p=2-\epsilon$. The full range $p >  \frac{n+3}{n+1}$  was established by Wolff in \cite{Wo}, while the end-point $p = \frac{n+3}{n+1}$ was established by Tao in \cite{Tao-BW} (both these results hold for all dimensions $n \geq 2$). 

For parabolic surfaces, Tao and Vargas established \eqref{ABRE} with $n=2$ and $p> 2-\frac2{17}$ in \cite{TV-CM1}, while 
the full range $p >  \frac{n+3}{n+1}$ (for all $n \geq 2$) was established by Tao in \cite{Tao-BP}. The end-point case 
$p =  \frac{n+3}{n+1}$ is still an open problem.

In \cite{Lee-BR}, Lee generalizes some of the above results to the case of surfaces with curvature of different signs.
In the quadratic case the model is $\tau=\sum_{i=1}^n \epsilon_i \xi_i^2, \epsilon_i \in \{-1,1\}$, while in the conic case the model is
$\tau \cdot \xi_n=\sum_{i=1}^{n-1} \epsilon_i \xi_i^2, \epsilon_i \in \{-1,1\}$. Lee establishes \eqref{ABRE} for $p>\frac{n+3}{n+1}$
under, apparently, stronger transversality assumptions between $S_1$ and $S_2$, where $S_1$ and $S_2$
are of the same type, that is both quadratic or both conic.    

In \cite{LeeVa}, Lee and Vargas, provide bilinear estimates for more general conic type surfaces, that is surfaces that have $k$-vanishing curvatures with $k \geq 1$. We note that the condition they impose on the conic surfaces with $k=1$ (one-vanishing curvature) is similar to our condition in the form \eqref{C3}. There is more to say about this work that goes beyond the formal statements made in it. 
After the first draft of our paper was made available, 
Lee and Vargas brought to our attention that the arguments they use in \cite{LeeVa} work for more general setups, and, in particular, they can cover our setup. However, their approach would use a combinatorics type argument, while ours uses an energy type argument. 

We learned very recently of some other directions of generalizations, see Buschenhenke, M\"uller, Vargas \cite{BuMuVa} and Stovall \cite{Sto}. 

In all the above works the role of transversality between $S_1$ and $S_2$ is clear. In the context of conic surfaces, the condition
is stated in terms of angular separation of $D_1$ and $D_2$, while for paraboloids the condition is stated in terms of separation of the domains $D_1$ and $D_2$.

However the precise role of the curvature in obtaining \eqref{ABRE} with optimal $p \geq p_0=\frac{n+3}{n+1}$ is not well understood. It is somehow disguised by the fact that all previous works have dealt with precise surfaces (or small perturbations of): conic-conic or quadratic-quadratic surface interactions. What is clear is that, in some sense, 
the optimal bilinear estimate relies on a lower count of the non-vanishing principal curvatures. Indeed, from the above it follows that one obtains the same 
result for parabolic or spherical surfaces (both having non-vanishing Gaussian curvature) as well as for
conic surfaces (which do have one vanishing principal curvature). The difference between these two types of surfaces is clear at the linear level in $R_S^*(2 \rightarrow p')$ as the quadratic surfaces yield an estimate with a lower $p$: $ p'_{quadratic}=\frac{2(n+2)}{n} < \frac{2(n+1)}{n-1}= p'_{conic}$ in the language used in \eqref{ALRE}.

We have come to the main point of this paper. Our goal is to obtain a universal theory for the bilinear estimate \eqref{ABRE} with optimal exponents $p > p_0=\frac{n+3}{n+1}$. This would require a complete understanding of the role of curvature in this problem and we do so by using geometric operators such as the shape operator or the second fundamental form. 

We consider two surfaces $S_1,S_2 \subset \R^{n+1}=\{(\xi,\tau): \xi \in \R^n, \tau \in \R\}$ that are graphs of smooth maps, that is 
they are given by the equations $\tau=\varphi_i(\xi), \xi \in D_i, i=1,2$, where $\varphi_1,\varphi_2$ are smooth in their domains $D_1,D_2 \subset \R^n$. 
The domains $D_1,D_2$ are assumed to have the usual properties: bounded, open, connected.
For each $i \in \{1,2\}$, we assume that $|\partial^\alpha \varphi_i(\xi)| \les 1$ in $D_i$, for sufficiently many multi-indeces $\alpha$. 
We use the identity map $id: \R^{n+1} \rightarrow \R^{n+1}$ to immerse $S_1, S_2 \subset \R^{n+1}$ as submanifolds.
We denote by $N_i(\zeta)= \frac{(-\nabla \varphi_i(\xi), 1)}{|(-\nabla \varphi_i(\xi), 1)|}$, where $\zeta=(\xi,\varphi_i(\xi)) \in S_i$, the normal to $S_i$ at the point $\zeta \in S_i$. 

Let $\tau^h$ denote the translation in $\R^{n+1}$ by the vector $h \in \R^{n+1}$, $\tau^h(x)=x+h, \forall x \in \R^{n+1}$. 
We denote by $C_2(h)= \tau^h(-S_1) \cap S_2 \subset S_2$ and $C_1(h)=-\tau^{-h}(C_2(h)) = S_1 \cap \tau^{h}(-S_2) \subset S_1$. 
In other words, $C_2(h)$ is obtained by intersecting $S_2$ with a translate of $-S_1$, while $C_1(h)$ is obtained by intersecting $S_1$ with a translate of $-S_2$. Note that $C_2(h)$ and $-C_1(h)$ are, modulo translations, a common submanifold of both $S_2$ and $-S_1$. 

For some $\zeta_i \in C_i(h) \subset S_i$ we split the tanget space $T_{\zeta_i} S_i = T_{\zeta_i} C_i(h) \oplus (T_{\zeta_i} C_i(h))^\perp$. For $\zeta_i \in S_i$, we denote by $S_{N_i(\zeta_i)}: T_{\zeta_i} S_i \rightarrow T_{\zeta_i} S_i$ the shape operator of the immersion $id : S_i \rightarrow \R^{n+1}$ (see Section \ref{RGC} for more detailed definitions). 

For $S_i, \tilde S_i$ defined as above, we say that $S_i \Subset \tilde S_i$ provided that $\bar D_i \subset \tilde D_i$, where we recall that $D_i, \tilde D_i$ are open subsets. In other words the inclusions $D_i \subset \tilde D_i$ and $S_i \subset \tilde S_i$ are compact.

For $u,v \in \R^{n+1}$ (thought as vectors) we define $|u \wedge v|$ to be the area of the parallelogram spanned by the vectors $u,v$, while for $u,v,w \in \R^{n+1}$ we define $vol(u,v,w)$ to be the volume of the parallelepiped spanned by the vectors $u,v,w$. 

We are now ready to state transversality and curvature properties of the surfaces involved in this paper.

\begin{con} (uniform transversality) We assume that $S_1$ and $S_2$ are transversal in the following sense:
\begin{equation} \label{C1}
| N_1(\zeta_1) \wedge N_2(\zeta_2)| \ges 1
\end{equation}
for any $\zeta_1 \in S_1, \zeta_2 \in S_2$ in an uniform way. 
\end{con}

\begin{con} (dispersion/curvature) We assume that for $i=1,2$ one of the following holds true for all $ h \in \R^{n+1}$
\begin{equation} \label{C3}
vol(N_i(\zeta_1) - N_i(\zeta_2), N_1(\tilde \zeta_1), N_2(\tilde \zeta_2)) \ges | \zeta_1 - \zeta_2 |, \quad \forall \zeta_1,\zeta_2 \in C_i(h), \forall \tilde \zeta_i \in S_i. 
\end{equation}
 \begin{equation} \label{C3aa}
|S_{N_i(\zeta_i)} v \wedge n| \ges | v | |n|, \quad \forall \zeta_i \in C_i(h), v \in T_{\zeta_i} C_i(h), n \in (T_{\zeta_i} C_i(h))^\perp.
\end{equation}
The bounds are meant to be uniform with respect to $h \in \R^{n+1}$. 
\end{con}

Since we will make multiple references to these two conditions, we will abbreviate them by using \bf C1\rm, respectively \bf C2\rm. 

It is worth noticing one difference between \eqref{C3} and \eqref{C3aa}: the former is global, while the latter is local. 
We will clarify in Section \ref{RGC} their local equivalence and why, for the purpose of the argument, the local/global aspect makes no difference.

At this point we can state the main result of this paper. 
\begin{theo} \label{mainT} Assume that $S_1 \Subset \tilde S_1, S_2 \Subset \tilde S_2$ where $\tilde S_1,\tilde S_2$ satisfy \bf C1 \it and \bf C2\it. For given $p$
with $\frac{n+3}{n+1} < p \leq 2$, the following holds true
\begin{equation} \label{mainE}
\| \widehat{f_1 d \sigma_1} \cdot \widehat{f_2 d\sigma_2} \|_{L^{p}(\R^{n+1})} \leq C(p, \tilde S_1, \tilde S_2)  \| f_1 \|_{L^2(S_1,d\sigma_1)}
\| f_2 \|_{L^2(S_2,d\sigma_2)}.
\end{equation}
\end{theo}

The use of $\tilde S_1, \tilde S_2$ should be understood as follows: we want $S_1$ and $S_2$ to satisfy \bf C1, C2\rm, but we also
want this to extend in some small neighborhoods of $S_1$ and $S_2$. This is because the argument uses at several stages re-localization
on both physical and Fourier side, and this potentially alters the support of the interacting components; in particular one cannot handle the argument 
with the rigid frequency localization in $S_1$ and $S_2$.

If $\rm II_{\it N_i(\zeta_i)}$ stands for the second fundamental form of $S_i$ at $\zeta_i \in S_i$ with respect to the normal $N_i(\zeta_i)$, then a slightly stronger variant of \bf C2 \rm is the following
\begin{con} (normal curvature)  We assume that for $i=1,2$ and $\forall h \in \R^{n+1}$ the following holds true
\begin{equation} \label{C3ab}
| \rm II_{\it N_i(\zeta_i)}\it(v,v)| \ges |v|^2,  \qquad \forall \zeta_i \in C_i(h), \forall v \in T_{\zeta_i} C_i(h). 
\end{equation}
The bound is meant to be uniform with respect to $h \in \R^{n+1}$. 
\end{con}

As a consequence of Theorem \ref{mainT} we obtain the following
\begin{corol} \label{corT} The result of Theorem \ref{mainT} holds true if $\tilde S_1,\tilde S_2$ satisfy \bf C1 \it and \bf C3 \it instead of \bf C1 \it and \bf C2\it.
\end{corol}

The above Corollary is our statement in terms of a standard curvature condition. It essentially says that, in addition to the transversality condition \bf C1\rm, 
if all curves in $C_i(h) \subset S_i$ have non-zero normal curvature, then the bilinear estimate \eqref{mainE} holds true. Given that $C_i(h)$ has dimension $n-1$, this result justifies why, in some sense, the optimal bilinear estimate relies only on $n-1$ "curvatures" being non-zero. One has to be careful in specifying 
which $n-1$ curvatures are meant to be non-zero: classically one would use the principal curvatures (see the commentaries below in the context of the $k$-linear restriction estimate), but \bf C3 \rm requires the stronger assumption that $n-1$
normal curvatures are non-zero. Even the more relaxed condition \eqref{C3aa} is stronger than asking $n-1$ principal curvatures to be non-zero. 

In making the above commentaries, we are implicitly saying that for given $\zeta_i \in S_i$, $\cup_h T_{\zeta_i} C_i(h) \ne T_{\zeta_i} S_i$, where $h$ varies such that $\zeta_i \in C_i(h)$. If this is not the case, it is an easy exercise to show that there are $\zeta_i \in S_i$ such that $T_{\zeta_1} S_1= T_{\zeta_2} S_2$ (where parallel tangent planes are trivially identified) which would contradict \bf C1\rm. As a consequence, it follows that there are curves $\gamma_i \subset S_i$ that are transversal to $C_i(h)$ for all $h$; potentially, among such curves we could find one, say $\gamma_i$, such that $\gamma_i(t_0)=\zeta_i$ and $\gamma'_i(t_0) \in T_{\zeta_i} S_i$ is a principal direction with zero eigenvalue, that is $S_{N_i(\zeta_i)} \gamma'_i(t_0)=0$. Indeed, this is the case for conic surfaces, but it is not for quadratic surfaces. 

In the next subsection we show that \bf C2 \rm and \bf C3 \rm are equivalent if $n=2$, while if $n \geq 3$, \bf C3 \rm implies \bf C2\rm, but not vice-versa. 

Another natural question to ask is the necessity of the two conditions. It is well known that in the absence of \bf C1\rm, no improvement of \eqref{ABRE} should be expected besides what follows from the linear estimates. In \cite{Lee-BR}, Lee gives an examples hinting that in the absence of \bf C3\rm, no improvement should be expected either. One needs to chase a bit this aspect in Lee's counterexample,
as his focus is on highlighting the fact that transversality (or domain separation) does not suffice for improvements in the bilinear estimate when one considers non-elliptic paraboloids. Oversimplifying Lee's example, essentially one considers in $\R^3$ the hyperbolic paraboloid $\tau=\xi_1^2 - \xi_2^2$, and notices that the embedded line $(t,t,0)$ has zero normal curvature. Letting $D_1,D_2$ be small neighborhoods
of $(1,1,0)$ and $(-1,-1,0)$ creates the counterexample. 
  
We now explain some of the key novelties this paper brings to the theory of bilinear estimates. In Harmonic Analysis, one way the shape operator plays a crucial role is through its eigenvalues which are the principal curvatures of the surface. As we explained earlier
in the context of $R^*_S(2 \rightarrow p')$, the role of the number of non-zero principal curvatures is well-understood in the linear theory. In this paper we reveal a more subtle way in which the shape operator affects the bilinear estimates which goes beyond the counting of its non-zero eigenvalues, see \eqref{C3aa} and next section, for details. To the best of our knowledge this may be the first instance in Harmonic Analysis when the shape operator enters the analysis of a problem in a more complex way, other than by its eigenvalues. 

In most previous works the equivalent of \bf C2 \rm was avoided by using explicit surfaces, see \cite{Wo, Tao-BW, Tao-BP}. 
What emerged in early works, see \cite{Lee-BR, Tao-BW, Va}, was the necessity of the result in Lemma \ref{LFW}, which is a consequence of our condition. In \cite{LeeVa} an equivalent condition to \eqref{C3} appears, see $(1.4)$ there. The condition \eqref{C3aa} and
the analysis of the role of the shape operator in the bilinear restriction problem is one of the new features in this paper.

At a technical level, the argument in our paper has to find a common ground for dealing with general surfaces. One of the reasons the role of the curvature in \eqref{ABRE} was not fully understood has had to do with the different methods used in dealing with the conic-conic and quadratic-quadratic cases. We summarize some of the key points which make our task possible.

\bf Wave packet theory. \rm The standard wave packet constructions for the quadratic and conic surfaces are slightly different, and this feature is not particular to the bilinear theory. It is commonly found in parametric construction via wave packets for the
Wave and Schr\"odinger equations. In this paper we use the same Wave packet construction for all hypersurfaces,  and this construction is dictated by the quadratic phase. This may be seen
as suboptimal for the conic surfaces or any non-quadratic surface, but it turns out that the geometry of the problem addresses  this issue in a very natural fashion.   

\bf Constant versus variable speed. \rm  In the standard approach for the conic surfaces, see \cite{Tao-BW}, the argument heavily relies on the fact that waves propagate with speed $1$. If $S_1$ and $S_2$ are conic surfaces, Tao explains in \cite{Tao-BP} that a key geometrical observation was that if one took the union of all the lines through a fixed origin $x_0$ which were normal to $S_2$, then any line normal to $S_1$ could only intersect this union in at most one point; this is ultimately due to the single vanishing principal curvature on the cone, which forces all of the above lines to be light rays. This property does not hold true for quadratic surfaces given the wider range of propagation speeds. Therefore a different type of argument was used for quadratic surfaces
see \cite{Tao-BP}. In our paper, the argument makes a very efficient use of orthogonality arguments and arranges the geometry of the problem to re-create 
the key geometrical observation just mentioned even in the case of quadratic surfaces, despite the variable speed of the wave in that setup; for details see Lemma \ref{LFW}. It is precisely this
part of the argument that brings out the natural conditions \bf C1 \rm and \bf C2\rm. 

\bf Energy versus combinatorics argument. \rm In proving bilinear estimates there are two main strategies: a combinatorial based approach
which has some qualitative aspects to it (by defining relations between tubes and balls) and an energy based approach which is "qualitative-free"
(this argument quantifies relations between tubes and balls using energy as a measure tool). The combinatorial approach is probably the most 
used, while the energy approach was developed by Tao in \cite{Tao-BW} with the scope of obtaining the end-point theory.
We prefer the latter one since it is more compact and has the advantage of tracking losses more carefully and potentially lay the ground 
for an end-point theory. Our paper draws inspiration from \cite{Tao-BW}, from which we use the notation and some technicalities.

We are not able to provide the end-point result $p=\frac{n+3}{n+1}$ for \eqref{ABRE}. The argument used in \cite{Tao-BW}
for proving the end-point estimate for conic surfaces uses in too many places the fact that waves propagate with speed $1$,
and we could not find a way to circumvent that aspect for general surfaces. Therefore the end-point problem
is still open and we think it suffices to understand it for the case when $S_1$ and $S_2$ are given by the elliptic paraboloid (model $\tau=|\xi|^2$), since it contains most of the difficulties. 

The lack of an end-point theory in the general case and its resolution for the conic case may suggests the following observation: in the context of bilinear estimates, additional curvature makes the problem more complicated. A more clear insight on the role of the curvature is revealed by looking at the $n+1$-multilinear estimate:
\begin{equation} \label{AMRE}
\| \Pi_{i=1}^{n+1} \widehat{f_i d \sigma_i} \|_{L^{p}(\R^{n+1})} \leq C \Pi_{i=1}^{n+1} \| f_i \|_{L^2(S_i,d\sigma_i)}.
\end{equation}
It is conjectured that if the hypersurfaces $S_i \subset \R^{n+1}$ are transversal, then \eqref{AMRE} holds true for $p \geq p_0=\frac{2}n$. 
If $S_i$ are transversal  hyperplanes, \eqref{AMRE}  is the classical Loomis-Whitney inequality and its proof is elementary.  Once the surfaces are allowed to have non-zero principal curvatures, things become far more complicated and the problem has been the subject of extensive research, see \cite{BeCaTa,Gu-main} and references therein. 
In \cite{BeCaTa}, Bennett, Carbery and Tao establish \eqref{AMRE} for with an $\epsilon$-loss in the following sense
\begin{equation} \label{AMRE-loss}
\| \Pi_{i=1}^{n+1} \widehat{f_i d \sigma_i} \|_{L^{p}(B(0,R))} \leq C(\epsilon) R^\epsilon \Pi_{i=1}^{n+1} \| f_i \|_{L^2(S_i,d\sigma_i)}.
\end{equation}
for any $\epsilon >0$; here $B(0,R)$ is the ball of radius $R$ centered at the origin. The result with a constant independent of $R$, that is \eqref{AMRE}, is an open problem. The end-point for the multilinear Kakeya version of \eqref{AMRE}
(a slightly weaker statement than \eqref{AMRE}) has been established by Guth in \cite{Gu-main} using tools from algebraic topology. 

The conclusion we wanted to draw from above is that a certain amount of curvature is needed in order to obtain the optimal bilinear restriction estimate \eqref{ABRE}, but that additional curvature brings complications to the problem. 

 We hope that the result of this paper will provide some insight into another open problem: $k$-multilinear estimates, these being estimates similar to \eqref{ABRE} and \eqref{AMRE}, but with $k$ terms, $1 \leq k \leq n+1$. Transversality between the surfaces involved is known to be a necessary condition for the optimality, thus we take it for granted and focus on the role of curvature in the discussion below. In \cite{BeCaTa}, the authors state that "simple heuristics suggest that the optimal $k$-linear restriction theory requires at least $n+1-k$ non-vanishing principal curvatures, but that further curvature assumptions have no further effect". 
Up to the present paper, the precise role of the curvature in the optimal $k$-linear estimate was fully understood only in the case $k=1$
and $k=n+1$: in the first case one needs all $n$ principal curvatures to be non-zero, while in the case $k=n+1$ no curvature is required. Our paper clarifies the role of the shape operator in the bilinear estimates, that is $k=2$, and the fact that information only about the principal curvatures does not suffice. Moreover, we believe that our setup provides the correct framework for making statements for the optimal $k$-linear restriction theory with $k \geq 3$, where the use of shape operator will probably be even more involved. We should mention that in the absence of any curvature assumptions, the $k$-linear restriction theory has been addressed in \cite{BeCaTa} where the authors establish it for $p > \frac{2}{k-1}$. However, if $k \leq n$, this is not expected to be the optimal exponent once curvature assumptions are brought into the problem.   

With the result of the present paper, the current optimal $k$-linear restriction theory covers in full only the cases  $k=1,2$ and $k=n+1$. 
Therefore it is only in the case $n=2$ (corresponding to transversal surfaces in $\R^3$) that the multilinear theory is now complete up to the end-point: 

- the case of one single surface is the classical Tomas-Stein result

- the case of bilinear estimates ($k=2$) and the role of curvature is provided in this paper

- the case of trilinear estimates where only transversality matters and curvature plays no role was settled in \cite{BeCaTa,Gu-main}.

In the case $n \geq 3$ and $3 \leq k \leq n$, the optimal multilinear theory is still an open problem. 

The multilinear theory discussed above has had major impact in other problems. We mention a few such examples we are aware of, 
but do not intend to provide a complete overview of applications or references. In Harmonic Analysis, the bilinear and 
$n+1$ multilinear theory was used to improve results in the context of Schr\"odinger maximal function, see \cite{Bou-SMF,Lee-SMF,TV-CM2, DL},  restriction conjecture, see \cite{Tao-BP,BoGu, Gu-res}, the decoupling conjecture, see \cite{BoDe} and \cite{BoDeGu}. In Partial Differential Equations, the linear theory inspired the well-known theory of Strichartz estimates which provides a fundamental tool for iterating dispersive equations, see \cite{Tao-book}. The bilinear restriction theory is used in the context of more sophisticated techniques, such as the profile decomposition, see for instance \cite{MeVe}, which is used in concentration compactness methods, see for instance \cite{KeMe}. And not last, we recall that the original conjecture about the optimal range for the bilinear estimate \eqref{ABRE} was motivated by the problem of improved bilinear estimates in the context of wave equations, see \cite{FoKl}.

The paper is organized as follows: in the next subsection we discuss \bf C2 \rm and highlight its main role in our argument. 
In Subsection \ref{Res} we derive all known results for \eqref{ABRE} from the results of Theorem \ref{mainT}, 
and show how new results are obtained. 
The Introduction ends with a Notation section in which we set some of the commonly used terminology. In Section \ref{CT} we restate the problem in terms of free waves, as it is commonly done in the literature, introduce the concept of tables on cubes and some basic results. In Section \ref{SE} we provide the energy estimate for waves traveling  through neighborhoods of surfaces to which they are transversal. In Section \ref{SWP} we provide the wave packet construction. Section \ref{SI} contains the induction on scale type argument, although a little hidden
into the table construction.

\subsection{Reading the geometric conditions} \label{RGC}
It is clear that \bf C1 \rm is the transversality condition. In this section we intend to shed more light into the nature of the conditions \bf C2, C3\rm.  Before doing so, we recall some basic facts from differential geometry that can be found in more detail in any classic differential geometry textbook, see for instance \cite{doCa}. 

In $\R^{n+1}$ (to be thought as its own tangent space at each point) the scalar product $\langle \cdot, \cdot \rangle$ is defined in the usual manner, the length of a vector is given by $|u|^2=\la u,u \ra$ and $|u \wedge v|= \sqrt{|u|^2 |v|^2-\la u,v \ra^2 }$ is the area of the parallelogram made by the vectors $u,v \in \R^{n+1}$.

Given a hypersurface  $S \subset \R^{n+1}=\{(\xi,\tau): \xi \in \R^n, \tau \in \R\}$ parametrized by $\tau=\varphi(\xi), \xi \in D \subset \R^n$, we define the Gauss map $g: S \rightarrow \S^{n} \subset \R^{n+1}$ ($\S^n$ is the unit sphere in $\R^{n+1}$) by 
$N(\zeta)=\frac{(-\nabla \varphi(\xi), 1)}{|(-\nabla \varphi(\xi), 1)|}$ where $\zeta=(\xi,\varphi(\xi)) \in S$. Since
$T_\zeta(S)$ and $T_{g(\zeta)} \S^{n}$ are parallel, we can identify them, and define $ d g_\zeta: T_\zeta S \rightarrow T_\zeta S$
by $d g_\zeta v= \frac{d}{dt} (N (\gamma(t)))|_{t=0}$ where $\gamma \subset S$ is a curve with $\gamma(0)=\zeta, \gamma'(0)=v$.  
The shape operator $S_{N(\zeta)}: T_\zeta S \rightarrow T_\zeta S$ is defined by 
\[
S_{N(\zeta)} = - d g_\zeta,
\]
 where we keep the subscript 
$N(\zeta)$ to indicate that the shape operator depends on the choice of the normal vector field at $S$. It is known that $S_{N(\zeta)}$ is symmetric, therefore there exists an orthonormal
basis of eigenvectors $\{e_i\}_{i=1,n}$ of $T_\zeta S$ with real eigenvalues $\{\lambda_i\}_{i=1,n}$. Locally $S$ is orientable and we assume
a consistency with the orientation in $\R^{n+1}$, that is $\{e_1,..,e_n\}$ is a basis in the orientation of $S$ and $\{e_1,..,e_n,N(\zeta)\}$ 
is a basis in the orientation of $\R^{n+1}$. Then $e_i$ are the principal directions and $\lambda_i=k_i$ are the principal curvatures of $S$ (to be more precise they are the curvatures of the embedding $id: S \rightarrow \R^{n+1}$, where $id$ is the identity mapping). The Gaussian curvature is defined by $det S_N=\Pi_{i=1}^n \lambda_i$. 

Finally, the second fundamental form $\rm II_{\it N(\zeta)}: \it T_\zeta S \rightarrow \R$ is defined by \footnote{This is not the usual order in which the objects are defined in differential geometry, but we do so in order to avoid a lengthier introduction; for details the reader is referred to \cite{doCa}, for instance.}
\[
\rm II_{\it N(\zeta)}\it(v)= \la S_{N(\zeta)} v, v \ra, \qquad v \in T_\zeta S.
\]

 We now aim to interpret condition \bf C2 \rm in the form \eqref{C3}. For simplicity we choose $i=1$ in \eqref{C3} and all statements we make below are valid for $i=2$ as well.
 Given that the normals are vectors of length  $1$, a consequence of \eqref{C3} is the following:
 \begin{equation} 
| N_1(\zeta_1) - N_1(\zeta_2) | \ges | \zeta_1 - \zeta_2 |, \qquad \forall \zeta_1,\zeta_2 \in C_1(h). 
\end{equation}
 This unveils a separation effect of the normals (to $S_1$) along $C_1(h)$, or a  dispersion effect along $C_1(h)$, to use a PDE language. 
 \eqref{C3} requires a stronger statement: the dispersion of the normals along $C_1(h)$ has to occur in directions that are transversal
 to the plane made by any two normals at the surfaces, that is any plane made by $N_1(\tilde \zeta_1)$ and $N_2(\tilde \zeta_2)$. 
 
 Next we show how \eqref{C3} implies \eqref{C3aa} and address the expected relation between the dispersion
 aspect of \eqref{C3} and the "non-zero curvature" of $S_1$ along $C_1(h)$ (note that this is loosely used here). 
 
 Let $g: S_1 \rightarrow \S^{n}, g(\zeta)=N_1(\zeta)$ be the Gauss map, 
 where $\S^{n} \subset \R^{n+1}$ is the unit sphere. We recall that $d g(\zeta)=-S_{N_1(\zeta)}$, where 
 $S_{N_1(\zeta)}: T_{\zeta} S_1 \rightarrow T_{\zeta} S_1$ is the shape operator. Passing to the limit $\zeta_2 \rightarrow \zeta_1$ with points inside $C_1(h)$
 in \eqref{C3} gives the following
 \begin{equation} \label{C3a}
vol(S_{N_1(\zeta_1)} v, N_1(\tilde \zeta_1), N_2(\tilde \zeta_2)) \ges | v |, \quad \forall \zeta_1 \in C_1(h), v \in T_{\zeta_1} C_1(h), \tilde \zeta_i \in S_i. 
\end{equation}
We let $\tilde \zeta_1=\zeta_1$ and $\tilde \zeta_2 = h - \zeta_1$ in the above. In this case both $N_1(\tilde \zeta_1), N_2(\tilde \zeta_2)$ are orthogonal 
to $T_{\zeta_1} C_1(h)$ and since they are linearly independent, they span the normal plane to $C_1(h) \subset \R^{n+1}$ at $\zeta_1$. 
Since $S_{N_1(\zeta_1)}: T_{\zeta_1} S_1 \rightarrow T_{\zeta_1} S_1$, \eqref{C3a} implies \eqref{C3aa}.
 
\eqref{C3aa} has two implications. First, $T_{\zeta_1} C_1(h)$ is transversal to the kernel of the shape operator $S_{N_1(\zeta_1)}$, in other words it is transversal
to any principal direction. Second, the shape operator $S_{N_1(\zeta_1)}$ cannot rotate tangent vectors in $T_{\zeta_1} C_1(h)$ into normal vectors in $(T_{\zeta_1} C_1(h))^\perp \subset T_{\zeta_1} S_1$.

If $n=2$, the last condition implies that $|\la S_{N_1(\zeta_1)} v, v \ra| = | \rm II_{\it N_1(\zeta_1)}\it(v,v)| \ges |v|^2$, where $\rm II_{\it N_1(\zeta_1)}$
is the second fundamental form at $\zeta_1$ along the normal $N_1(\zeta_1)$. In other words, $C_1(h)$ (which is a curve) has to have non-zero normal curvature, and 
in a more precise fashion, its normal curvature $\kappa_n$ satisfies $|\kappa_n| \ges 1$. 

In higher dimensions, this interpretation would provide a sufficient, but not necessary condition. 
If $|\la S_{N_1(\zeta_1)} v, v \ra| =| \rm II_{\it N_1(\zeta_1)}\it(v,v) | \ges |v|^2$, then \eqref{C3aa} holds true, thus we have a

\begin{proof}[Proof of Corollary \ref{corT}] 
 \eqref{C3ab} implies \eqref{C3aa} and we can apply Theorem \ref{mainT}.
  \end{proof}

However, given that $T_{\zeta_1} C_1(h)$ has dimension at least two, it is possible to have $v_1,v_2 \in T_{\zeta_1} C_1(h)$, $v_1$ and $v_2$ orthogonal to each other, such that $S_{N_1(\zeta_1)} v_1=v_2$. In this case $ \rm II_{\it N_1(\zeta_1)}\it(v_1,v_1)=0$, but \eqref{C3a} is not violated.
 
 We proved that \eqref{C3aa} is a consequence of \eqref{C3}, and a natural question is whether they are equivalent. The equivalence holds locally:
 parts of \eqref{C3a} were obtained by passing to limits, thus the reverse process holds locally; inferring \eqref{C3a} with general 
 $\tilde \zeta_1 \in S_1, \tilde \zeta_2 \in S_2$ from \eqref{C3aa}  can be done provided that the Gauss map has very small
 variations, that is $|N_i(\tilde \zeta)- N_i(\zeta)| \ll 1, \forall \tilde \zeta,\zeta \in S_i, i=1,2$ and this holds true locally. 
 
The fact that the equivalence of \eqref{C3} and \eqref{C3aa} is guaranteed locally only does not affect the result of Theorem \ref{mainT}. Indeed, one can break $S_1$ and $S_2$ into smaller pieces where the equivalence holds, use the Theorem \ref{mainT} for these pieces and than add back the estimates for the smaller pieces to obtain the global estimate. Based on this observation, for the rest of the paper we will use the following additional hypothesis
\begin{equation} \label{Nvar}
|N_i(\zeta_i) - N_i(\zeta_i^0)| \ll 1, \quad \forall \zeta_i \in S_i, \quad i=1,2, 
\end{equation}
where $\zeta_i^0 \in S_i$ is some fixed point.

 While \eqref{C3a} brings in natural objects such as the shape operator used to measure the amount of curvature, it lacks the computational 
 advantage that a formulation in terms of the Hessian of $\varphi_1$ and $\varphi_2$ would have. For instance on the very simple models 
 $\varphi_1(\xi)=\sum_{i} c_i \xi_i^2$, where the $H\varphi_1$ is a constant matrix, the shape operator involves doable but complicated computations. 
 
 Fix $i=1$. From the formula giving the normals $N_i(\zeta)= \frac{(-\nabla \varphi_i(\xi), 1)}{|(-\nabla \varphi_i(\xi), 1)|}$ and \eqref{Nvar}, it follows that \eqref{C3} is equivalent to 
 \begin{equation} 
vol((\nabla \varphi_1(\xi_1) - \nabla \varphi_1(\xi_2),0), (-\nabla \varphi_1(\tilde \xi_1),1), (-\nabla \varphi_2 (\tilde \xi_2),1) \ges | \xi_1 - \xi_2 |,  
\end{equation}
 $\quad \forall \xi_1,\xi_2 \in \Pi_1 C_1(h), \tilde \xi_i \in D_i$ where $\Pi_1(C_1(h)) \subset \R^n$ be the "projection" of $C_1(h)$ onto $D_1 \subset \R^n$, that is 
$C_1(h)=\{ (\xi,\varphi_1(\xi)); \xi \in \Pi_1(C_1(h)) \}$. Passing to the limit $\xi_2 \rightarrow \xi_1$ gives
\[
vol((H \phi_1(\xi_1) v,0), (-\nabla \varphi_1(\tilde \xi_1),1), (-\nabla \varphi_2 (\tilde \xi_2),1) \ges | v |,
\]
for any $v \in T_{\xi_1} \Pi_1(C_1(h))$. Given that $(H\varphi_1(\xi_1) v,0)$ is transversal to $(0,0,1)$ and 
$|H \varphi_1(\xi_1) v| \approx |v|$ (consequence of the above inequality), the above condition implies
\[
|  H \varphi_1(\xi_1) v \wedge ( \nabla \varphi_1(\tilde \xi_1)-\nabla \varphi_2 (\tilde \xi_2) ) | \ges |v|
\]
 We need to make sense of the meaning of the term $\nabla \varphi_1(\tilde \xi_1)-\nabla \varphi_2 (\tilde \xi_2)$.
Given some $\tilde h=(\tilde h_0,\tilde h_{n+1}), \tilde h_0 \in \R^n$, we have that $\Pi_1 C_1(\tilde h)$ is given by the equation
\[
\varphi_1(\xi_1) + \varphi_2(\tilde h_0 -\xi_1) = \tilde h_{n+1}. 
\]
The normal to $\Pi_1 C_1(\tilde h)$ is given by $\nabla \varphi_1(\xi_1) - \nabla \varphi_2(\tilde h_0 -\xi_1)$. Given $\tilde \xi_1 \in D_1, \tilde \xi_2 \in D_2$
we can find $\tilde h$ such that $\nabla \varphi_1(\tilde \xi_1)-\nabla \varphi_2 (\tilde \xi_2))$ is the normal to $C_1(\tilde h)$.  The transversality condition \eqref{C1} implies that $|\nabla \varphi_1(\tilde \xi_1)-\nabla \varphi_2 (\tilde \xi_2))| \ges 1$. Hence the above conditions reads
\begin{equation} \label{C3bb}
|  H \varphi_1(\xi_1) v \wedge n  | \ges |v| |n|, \qquad \forall v  \in T_{\xi_1} \Pi_1(C_1(h), n \in T_{\xi_1}^\perp \Pi_1(C_1(h)), 
\end{equation}
From this we derive two conclusions:  $T_{\xi_1} \Pi_1 C_1(h)$ is transversal to the kernel of the Hessian $H \varphi_1(\xi_1)$.
 $H \varphi_1(\xi_1)$ cannot rotate tangent vectors in $T_{\xi_1} \Pi_1 C_1(h)$ into normal vectors to 
$(T_{\xi_1} \Pi_1 C_1(h))^\perp \subset \R^n$.
 
 The resemblance of \eqref{C3bb} with \eqref{C3aa} is not accidental. If $\nabla \varphi_1(\xi_0)=0$, then $N_1(\zeta_0)=(0,..,0,1)$, 
 $T_{\zeta_0} S_1= \{(v,0): v \in \R^n\}$ and $S_{N_1(\zeta_0)}=H \varphi_1(\xi_0)$ with $S_{N_1(\zeta_0)} (v,0) = H \varphi_1(\xi_0) v$. 

As argued earlier, the equivalence of \eqref{C3} or \eqref{C3aa} with \eqref{C3bb} holds locally.

We now explain the practical consequences of \bf C1 \rm and \bf C2\rm. 
Given $C_1(h)$ defined as above, let $\mathcal{CN}(C_1(h))=\{\alpha N_1(\zeta), \zeta \in C_1(h), \alpha \in \R \}$
be the cone generated by the normals to $S_1$ taken at points from $C_1(h)$ and passing through the origin. 
Note that $\mathcal{CN}(C_1(h)) \setminus \{0\}$ has maximal codimension $1$. This hypersurface has one property which will play a crucial role in our argument. For any $\zeta_2 \in S_2$, $N_2(\zeta_2)$ is transversal to each $N_1(\zeta_1)$, for any $\zeta_1 \in S_1$  (consequence of \bf C1\rm). However, this does not imply that $N_2(\zeta_2)$ is transversal to the surface $\mathcal{CN}(C_1(h))$! Such a claim is the object of the following result. 

\begin{lema} \label{LFW}
For any $\zeta_2 \in S_2$, $N_2(\zeta_2)$ is transversal to the cone $\mathcal{CN}(C_1(h)) \setminus \{0\}$. Therefore a line in the direction of $N_2(\zeta_2)$, for some $\zeta_2 \in S_2$, intersects $\mathcal{CN}(C_1(h))$ locally at most in one point.

\end{lema}

Since the conditions \bf C1, C2 \rm are symmetric with respect to $S_1,S_2$, the above result is also symmetric: 
$N_1(\zeta_1)$ is transversal to the cone $\mathcal{CN}(C_2(h))  \setminus \{0\}$. 

\begin{proof}  Consider $\zeta_1 \in C_1(h)$ and let $\zeta_2=-\zeta_1 + h \in C_2(h) \subset  S_2$. We first prove the result for this choice of $\zeta_2$. The plane spanned by $N_1(\zeta_1)$ and $N_2(\zeta_2)$ is orthogonal to $T_{\zeta_1} C_1(h)$.

We prove that $N_2(\zeta_2)$ is transversal to $T_{\zeta_1}(\mathcal{CN}(C_1(h)))$, the tangent plane to $\mathcal{CN}(C_1(h))$ at $\zeta_1$. Since $\mathcal{CN}(C_1(h))$ is a conic surface, its tangent space at the point $\alpha N_1(\zeta_1), \alpha \in \R \setminus \{0\}, \zeta_1 \in C_1(h)$, is spanned by $N_1(\zeta_1)$ and the linear subspace 
$d g(\zeta_1) T_{\zeta_1} C_1(h)=\{d g(\zeta_1) v= - S_{N_1(\zeta_1)} v:v \in T_{\zeta_1} C_1(h)\}$. We recall that \eqref{C3aa} implies that this linear subspace is transversal to the plane spanned
by $N_1(\zeta_1)$ and $N_2(\zeta_2)$. Since $N_1(\zeta_1)$ and $N_2(\zeta_2)$ are transversal to each other, we conclude that $N_2(\zeta_2)$ is transversal to the subspace spanned by $N_1(\zeta_1)$ and $d g(\zeta_1) T_{\zeta_1} C_1(h)$, thus it is transversal to $T_{\zeta_1}(\mathcal{CN}(C_1(h)))$. 

For an arbitrary $\zeta_2 \in S_2$ the same conclusion follows in light of \eqref{Nvar}.  

A similar proof can be made starting from \eqref{C3} instead.

\end{proof}

\subsection{Consequences of Theorem \ref{mainT}} \label{Res}
Here we explain how previous results follow as consequences of Theorem
\ref{mainT}, as well as how new results can be derived from it. Consider the case when the two surfaces are of quadratic type, 
that is $\varphi_i(\xi)=\sum_{k=1}^n c_k^i \xi_k^2$ with $c_k^i \ne 0, \forall k=1,..,n, i=1,2$. 
If for each $i=1,2$, all $c_k^i, k=1,..,n$ have the same sign, then \eqref{C3ab} holds true for any $v \in T_{\zeta_i} S_i$, therefore the transversality condition \bf C1 \rm is the only one required. But this amounts to the separation of the domains $D_1,D_2$, that is $dist(D_1,D_2) > 0$. This implies the result in \cite{Tao-BP}. 

Next, consider the case when $c_i$'s have variable signs or more generally when $H \varphi_i$ is nonsingular with eigenvalues of different signs. The domain separation $dist(D_1,D_2) > 0$ suffices to ensure \bf C1\rm, but \bf C2 \rm is not true everywhere. 
Given that $S_{N_i(\zeta_i)}$ is non-singular, \eqref{C3aa} holds true provided that
\[
| \la S^{-1}_{N_i(\zeta_i)} n_{\zeta_i}, n_{\zeta_i} \ra| \ges 1, \qquad i=1,2
\]
where $n_{\zeta_i} \in T_{\zeta_i}(C_i(h))^\perp \subset T_{\zeta_i} S_i$ is the unit normal to $C_i(h) \subset S_i$. In local coordinates, this becomes 
(in light of \eqref{C3bb})
\begin{equation} \label{CLee}
| \la H^{-1} \varphi_i(\xi_i) n , n \ra | \ges 1, 
\end{equation}
where $n \in T_{\xi_i}^\perp \Pi_i(C_i(\tilde h)), |n|=1$. This last formulation is, essentially, the one found in \cite[ Theorem 1.1]{Lee-BR}. The above condition has the advantage of being somehow more compact, but it lacks a clear geometrical meaning. On the other hand, \eqref{C3aa}
is more transparent: if one avoids having curves of zero normal curvature in $S_1$ and $S_2$ then the result holds true; however, one can allow curves 
of zero normal curvature in $S_1$ and $S_2$ provided \eqref{C3aa} holds true (this can happen only if $n \geq 3$).

Next we consider conic surfaces given by $\tau \cdot \rho = \la \eta, H_i \eta \ra$, where $H_i$ are non-singular $(n-1) \times (n-1)$ matrices. Rescaling the equation we obtain $ \tilde \tau = \la \tilde \eta, H_i \tilde \eta \ra=\varphi_1(\tilde \eta)$ where $\tilde \tau=\frac{\tau}{\rho}, \tilde \eta= \frac{\eta}{\rho}$. 
In these new variables we are dealing with the setup similar to the previous one in the quadratic case. Therefore for the classical case of the cone, that is 
$H_1=H_2=I_{n-1}$ (identity matrix), only \bf C1 \rm needs to be imposed and it can be easily shown that the domain separation for
the new variables $\tilde \eta$ corresponds to the standard angular separation condition for the domains as used in \cite{Wo, Tao-BW}. 
For the case of mixed signs, that is $H_1,H_2$ are diagonal matrices with non-zero entries, but variable sign, then one simply uses the above discussions
for the quadratic case. In particular, if all entries are $\pm 1$, then $H_i^{-1}=H_i$ and \eqref{CLee} implies the following 
\[
| \la H_i n , n \ra | \ges 1.
\]
Given that the normals are obtain as follows $n=\nabla \varphi_i ( \tilde \eta_1) - \nabla \varphi_i ( \tilde \eta_2)= H (\tilde \eta_1 -\tilde \eta_2)$, 
with $\tilde \eta_1 \in \tilde D_1, \tilde \eta_2 \in \tilde D_2$, and $H^2=I_{n-1}$ it is easy to see that the above condition corresponds to 
\cite[ Theorem 1.3]{Lee-BR}. 

A new application is the case of mixed surfaces. To keep things simple let $S_1$ be the standard paraboloid $\tau=\varphi_1(\xi)=|\xi|^2$ and $S_2$ be the standard cone $\tau=\varphi_2(\xi)=|\xi|$. We assume $D_1,D_2$ are subsets of neighborhoods of the origin, with $0 \notin D_2$. We claim that the condition \bf C2 \rm is satisfied without any additional assumptions on $D_1$ and $D_2$. Indeed, given that $H\varphi_1=I_n$, \bf C2 \rm holds true on $S_2$. 
As for $S_1$, \bf C2 \rm fails to hold true if $C_2(h)$ contains straight lines; but this is impossible since $-\tau^{-h}(C_2(h)) \subset C_1(h) \subset S_1$
and there are no straight lines in $S_1$. Therefore, we only need to verify \bf C1\rm. If $||\nabla \varphi_1(\xi_1)|-1| \ges 1, \forall \xi_1 \in D_1$, then the transversality condition is fulfilled. If inside $D_1$ there are points with $|\nabla \varphi_1 (\xi_1)|=1$, then an angular separation condition between 
$\tilde D_1=\{ \xi_1: ||\nabla \varphi_1(\xi_1)|-1| \ll 1 \}$ and $D_2$ is required. This argument is easily extended to more general $\varphi_1$ as long as $H \varphi_1$ is non-singular with all eigenvalues having the same sign. 

But there is also a higher degree of generality in our result. The surfaces we consider are allowed to have one direction where
the degree of contact $k$ with the tangent plane satisfies $2 < k < \infty$, the simplest model being $\varphi(\xi)=\xi_1^k+\xi_2^2+..+\xi_n^2$ at the origin. In all previous works, the degree was either $k=2$ (quadratic-quadratic) or $k=\infty$ (conic-conic). 

As we have already discussed, a necessary condition for \bf C2 \rm to hold true is that $S_1$ and $S_2$ have each at most one zero principal curvature. The theory we developed here can be extended to the case when $S_1$ and $S_2$ have less than $n-1$ non-zero principal curvatures. It is interesting to notice that our argument is able to read faithfully the different ways in which \eqref{C3aa} fails:
$v$ is an eigenvalue of $S_{N_i(\zeta_i)}$ versus $v$ is rotated by $S_{N_i(\zeta_i)}$, 
that is $| \la S_{N_i(\zeta_i)} v, n \ra| \ges | v | |n|$. In the first case dispersion in the corresponding direction is completely absent, 
while in the second case dispersion occurs but in the non-optimal direction. It is the energy estimate in Section \ref{SE} which discriminates between the two cases and will lead to different results. However, given that, in such a situation, the lower bound for $p$ in \eqref{ABRE} would become higher than the optimal $p_0=\frac{n+3}{n+1}$, we do not pursue this issue in this paper. 

\subsection{Notation}

We now explain the use of various constants that appear throughout the rest of the argument. 
$N$ is a large integer that depends only on the dimension. $C$ is a large constant that may change
from line to line, may depend on $N$, but not on $c$ and $C_0$ introduced below. $C$ is used in the definition 
of: $A \ls B$, meaning $A \leq C B$,  $A \ll B$, meaning $A \leq C^{-1} B$, and $A \approx B$, meaning $A \les B \wedge B \les A$.
For a given number $r \geq 0$, by $A = O(r)$ we mean that $A \approx r$. 

$C_0$ is a constant that is independent of any other constant and its role is to reduce the size of cubes in the inductive argument. 
It can be set $C_0=4$ throughout the argument, but we keep it this way so that its role in the argument is not lost.

Finally, $c \ll 1$ is a very small variable meant to make expressions $\ll 1$ and most estimates will be stated to hold in a range of $c$.

By powers of type $R^{\alpha+}$ we mean $R^{\alpha+\epsilon}$ for arbitrary $\epsilon > 0$. Practically they should be seen
as $R^{\alpha+\epsilon}$ for arbitrary $0 < \epsilon \ls 1$. The estimates where such powers occur will obviously depend on $\epsilon$. 

Let $\eta_0:\R^n \rightarrow [0,+\infty)$ be a Schwartz function, normalized in $L^1$, that is $\| \eta_0 \|_{L^1}=1$,
and with Fourier transform supported on the unit ball. 

A disk $D \subset \R^{n+1}$ has the form
\[
D=D(x_D,t_D;r_D)=\{(x,t_D) \in \R^{n+1}: |x-x_D| \leq r_D\},
\]
for some $(x_D,t_D) \in \R^{n+1}$ and $r_D > 0$. We define the associated smooth cut-off
\[
\tilde \chi_{D}(x,t)= (1+\frac{|x-x_D|}{r_D})^{-N}. 
\] 

A cube $Q \subset \R^{n+1}$ of size $R$ has the standard definition $Q=\{(x,t) \in \R^{n+1}: \|(x-x_Q,t-t_Q) \|_{l^\infty} \leq \frac{R}2 \}$, where $(x_Q,t_Q)$ is the center of the cube. Given a constant $\alpha >0$ we define $\alpha Q$ to be the dilated by $\alpha$ of $Q$ from its center, that is
$\alpha Q=\{(x,t) \in \R^{n+1}: \|(x-x_Q,t-t_Q) \|_{l^\infty} \leq \alpha \cdot \frac{R}2 \}$.

\section{Restating the problem} \label{CT}
 
 \subsection{Rephrasing the problem in terms of free waves} Here we reformulate our problem
in terms of free waves, this being motivated by the use of wave packets in order to prove Theorem \ref{mainT}.
The setup used in this section follows closely \cite{Tao-BW}. 

We parametrize the physical space by $(x,t) \in \R^{n} \times \R$. In what follows we use the convention that $\hat f$ denotes the Fourier transform of $f$ with respect to the $x$ variable, while $\widehat{f}$ denotes the Fourier transform of $f$ with respect to the $(x,t)$ variable. In most cases it will be clear from the context 
which Fourier transform is used. 

We define the free wave $\phi(x,t) = \widehat{f_1 d \sigma_1}(x,t)$ as follows
 \[
 \phi(x,t) = \widehat{f_1 d \sigma_1}(x,t)= \int_{S_1} e^{i(x,t) \cdot z} f_1(z) d\sigma_1(z) = \int_{\R^n} e^{i(x\cdot \xi + t \varphi_1(\xi))} \hat \phi_0(\xi) d\xi
 \]
 where $\hat \phi_0(\xi):= f_1(\xi,\varphi_1(\xi)) \sqrt{1+|\nabla \varphi_1(\xi)|^2}$ satisfies $\| \phi_0 \|_{L^2(\R^n)} \approx \| f_1 \|_{L^2(S_1,d\sigma_1)}$. 
 From the above it follows that $\hat \phi(\xi,t) = e^{it \varphi_1(\xi)} \hat \phi_0(\xi)$ therefore $\phi$ satisfies an ODE on the Fourier side, $\partial_t \hat \phi(\xi,t) = i \varphi_1(\xi) \hat \phi(\xi,t)$, and a linear PDE on the physical side, $\partial_t \phi = i \varphi_1(\frac{D}{i}) \phi$ with initial data $\phi(x,0)=\phi_0(x)$.
 This justifies the wording: $\phi$ is a free wave.  
 
 We define the mass of a free wave by $M(\phi(t)):= \| \phi(t) \|^2_{L^2}$ and note that it is time independent:
 \[
 M(\phi(t)):= \| \phi(t) \|^2_{L^2}= \| \hat \phi(t) \|^2_{L^2} = \| \hat \phi_0 \|^2_{L^2}= \| \phi_0 \|^2_{L^2}= M(\phi_0). 
 \]
 It is clear from its definition that $\widehat{\phi}(\xi,\tau)$ is supported on $S_1$  given by $\tau=\varphi_1(\xi)$. In fact, in order to
 have concise notation, when referring to such $\phi$'s, we will abuse notation and say that $\phi$ is a free wave with $\widehat \phi$ supported on $S_1$. 
 
 In a similar manner we define the free wave $\psi(x,t) = \widehat{f_2 d \sigma_2}(x,t)$ and introduce $\psi_0$ by 
 $\hat \psi_0(\xi):= f_2(\xi,\varphi_2(\xi)) \sqrt{1+|\nabla \varphi_2(\xi)|^2}$ satisfying $\| \psi_0 \|_{L^2} \approx \| f_2 \|_{L^2(S_2,d\sigma_2)}$.
 
With these new entities, the result of Theorem \ref{mainT} follows from the following claim:
if $\phi,\psi$ are two free waves with Fourier transform supported on $S_1, S_2$
respectively, the following holds true:
\begin{equation} \label{mainEW}
\| \phi \psi \|_{L^p(\R^{n+1})} \les M(\phi)^\frac12 M(\psi)^\frac12. 
\end{equation}

The proof of \eqref{mainEW} relies on estimating $\phi \psi$ on cubes on the physical side and see how this 
behaves as the size of the cube goes to infinity by using an inductive type argument with respect to the size of the cube.
Before we formalize this strategy, we note that at every stage of the inductive argument we re-localize functions both on the physical and frequency space, and, as a consequence, we need to quantify the new support on the frequency side. This will be done by using the using the margin of a function. 
Let $M=\min(\mbox{dist}(D_1,\tilde D_1^c), \mbox{dist}(D_2,\tilde D_2^c))$, where the complements of $\tilde D_1^c$ and $\tilde D_2^c$ are taken in $\R^n$. 

For a function $f(x,t)$ we define the margin 
\[
\mbox{margin}^k(f(t)) := \mbox{dist}(\mbox{supp}_\xi (\hat f(t)), \tilde D_k^c), \quad k=1,2, 
\]
where $\mbox{supp}_\xi $ is the support with respect to the $\xi$ variable of $\hat f$.
Note that the frequency support of a free wave is the same for all times, therefore its margin is time independent. 

\begin{defin} Let $p_0 \leq p \leq 2$. Given $R \geq C_0$ we define $A_p(R)$ to be the best constant for which the estimate
\begin{equation}
\| \phi \psi \|_{L^p(Q_R)} \leq A_p(R) M(\phi)^\frac12 M(\psi)^\frac12
\end{equation}
holds true for all cubes $Q_R$ of size-length $R$, $\phi,\psi$ free waves with $\widehat \phi, \widehat \psi$ supported on $S_1,S_2$, respectively, and 
obeying the margin requirement
\begin{equation} \label{mrpg}
margin^1(\phi), margin^2(\psi) \geq M-R^{-\frac14}.
\end{equation}
\end{defin}

The goal is to obtain an uniform estimate on $A_p(R)$ with respect to $R$. In the absence of the margin requirement above, 
$A_p(R)$ would be an increasing function. However, since the argument needs to tolerate 
the margin relaxation, we also define
\[
\bar A_p(R):= \sup_{1 \leq r \leq R} A_p(r) 
\]
and the new $\bar A_p(R)$ is obviously increasing with respect to $R$.  

Then \eqref{mainEW}, and as a consequence the main result of this paper, Theorem \ref{mainT}, follow from the next result. 
\begin{prop} \label{keyP} If $R \gg 2^{2C_0}$ and $R^{-\frac14+} \ll c \ll 1$, the following holds true:
\begin{equation} \label{ApR}
A_p(R) \leq (1+cC) \bar A_p(\frac{R}2) + C c^{-C} R^{\frac{n+3}2(\frac1p-\frac{n+1}{n+3})}.
\end{equation}
\end{prop}

Now we show how \eqref{mainEW} follows from \eqref{ApR}. 
Since $p > \frac{n+3}{n+1}$, we set $c^{-C}=R^{-\frac{n+3}4(\frac1p-\frac{n+1}{n+3})}$, 
that is $c = R^{\frac{n+3}{4C}(\frac1p-\frac{n+1}{n+3})}$, and note that $c$ satisfies $R^{-\frac14+} \ll c \ll 1$, provided $C(n,p)$ is large enough.
Then we apply \eqref{ApR} to obtain
\[
A_p(R) \leq (1+C R^{\frac{n+3}{4C}(\frac1p-\frac{n+1}{n+3})}) \bar A_p(\frac{R}2) + C R^{\frac{n+3}4(\frac1p-\frac{n+1}{n+3})}.
\]
Taking the maximum with respect to $r \in [\frac{R}2,R]$ gives
\[
\bar A_p(R) \leq (1+C R^{\frac{n+3}{4C}(\frac1p-\frac{n+1}{n+3})}) \bar A_p(\frac{R}2) + C R^{\frac{n+3}4(\frac1p-\frac{n+1}{n+3})}.
\]
Since both powers of $R$ are negative, 
$\frac{n+3}{4C}(\frac1p-\frac{n+1}{n+3}), \frac{n+3}4(\frac1p-\frac{n+1}{n+3}) < 0$, this estimate can be iterated
to show that $\bar A_p(R)$ is uniformly bounded in terms of $\bar A_p(C2^{2C_0})$ for all $R \geq C2^{2C_0}$. Since $\bar A_p(C2^{2C_0})$
is bounded by a constant depending on $C_0$ and $C$, \eqref{mainEW} follows and we conclude the proof of Theorem \ref{mainT}. 
 
\subsection{Tables on cubes}
 
Let $Q \subset \R^{n+1}$ be a cube of radius $R$. Given $j \in \N$ we split $Q$ into $2^{(n+1)j}$ cubes of size $2^{-j} R$ and denote this family by $\calQ_j(Q)$; thus we have $Q=\cup_{q \in \calQ_j(Q)} q$.  
If $j \in \N$ and $0 \leq c \ll 1$ we define the $(c,j)$ interior $I^{c,j}(Q)$ of $Q$ by
\begin{equation} \label{icj}
I^{c,j}(Q) := \bigcup_{q \in \calQ_j(Q)} (1-c) q.
\end{equation}
Given $j \in \N$ we define a table $\Phi$ on $Q$ to be a vector $\Phi=(\Phi^{(q)})_{q \in \calQ_j(Q)}$ and define its mass by
\[
M(\Phi) = \sum_{q \in \calQ_j(Q)} M(\Phi^{(q)}). 
\]
We define the margin of a table as the minimum margin of its components:
\[
margin(\Phi) = \min_{q \in \calQ_j(Q)} margin(\Phi^{(q)}).
\]

Inspired by the Lemma 6.1 in \cite{Tao-BW}, we will make use of the following result. 
\begin{lema} \label{AVL}
Assume $R \gg 1$, $0 < c \ll 1$ and $f$ smooth. Given a cube $Q_R \subset \R^{n+1}$ of size $R$,
there exists a cube $Q$ of size $2R$ contained in $4 Q_R$ such that
\begin{equation} \label{avrg}
\| f \|_{L^p(Q_R)} \leq (1+cC) \|  f\|_{L^p(I^{c,j}(Q))}
\end{equation}
\end{lema}
\begin{proof} Using Fubini's theorem, we have the following identity
\[
\int_{Q_R} \| f \|^p_{L^p((Q_R \cap I^{c,j}(Q(x,t;2R)))} dx dt= \int_{Q_R} |f(x,t)|^p |Q_R \cap I^{c,j}(Q(x,t;2R))| dx dt.
\]
From the definition of $I^{c,j}(Q(x,t;2R))$ it follows that
\[
|Q(x,t;2R) \setminus I^{c,j}(Q(x,t;2R))| \leq (n+1) c |Q(x,t;2R)| = (n+1) 2^{n+1} c |Q_R|
\]
and, as a consequence,
\[
|Q_R| \leq (1+(n+1)2^{n+1}c) | Q_R \cap  I^{c,j}(Q(x,t;2R)) |, \qquad \forall (x,t) \in Q_R.
\]
In the above we have used that if $(x,t) \in Q_R$ then $Q_R \subset Q(x,t;2R)$. 

Combining this estimates with the above identity, leads to
\[
\| f \|_{L^p}^p \leq \frac1{|Q_R|} \int_{Q_R} (1+(n+1)2^{n+1}c) \| f \|^p_{L^p((Q_R \cap I^{c,j}(Q(x,t;2R)))} dx dt
\]
By the pigeonholing principle, it follows that there is $(x,t) \in Q_R$ such that
\[
\| f \|^p_{L^p} \leq (1+(n+1)2^{n+1}c) \| f \|^p_{L^p((Q_R \cap I^{c,j}(Q(x,t;2R)))}
\]
and since $(1+(n+1)2^{n+1}c)^\frac1p \leq 1+cC $, the conclusion follows.
\end{proof}

\section{Energy estimates across a neighborhood of a surface} \label{SE}

In this section we provide energy estimates in neighborhoods of the conic surfaces defined in the introduction. 
We set $i=1$ and recall the definition of $\mathcal{CN}(C_1(h))=\{\alpha N_1(\zeta), \zeta \in C_1(h), \alpha \in \R \}$, the conic surface generated by the normals to $C_1(h) \subset S_1$ passing through the origin. For a given surface $S \subset \R^{n+1}$ we denote the neighborhood of size $r$ of $S$ by $S(r)$. For fixed $t$ we define the time "slice" in $S(r)$ by $S_t(r)=\{ x: (x,t) \in S(r) \}$.

\begin{lema} Let $\psi$ be a free wave with $\hat \psi$ supported on $S_2$. Let $S=\mathcal{CN}(C_1(h))$.
We assume that for  any $\zeta \in S_2$, the vector $N_2(\zeta)$
 is transversal to $S$ in a uniform fashion. If $r \ges 1$, the following holds true:
\begin{equation} \label{ES}
\| \psi \|_{L^2(S(r))} \ls r^\frac12 M(\psi)^\frac12. 
\end{equation}
\end{lema}

Note that is $S$ were a planar surface, then the above estimate would follow from the standard energy
estimates for $\psi$ in various coordinate systems, a topic well studied in PDE's, see for instance \cite{bikt,tat} and references therein. In the current setup, the surface $S$ is more general, in particular it can "curve" if $C_1(h)$ has nonzero curvatures (and this will be indeed the case given our hypothesis).

\begin{proof} \eqref{ES} is equivalent to
\[
\| \chi_{S(r)} e^{it\varphi_2(D)} \psi_0  \|_{L^2(\R^{n+1})} \ls  r^\frac12 \| \psi_0 \|_{L^2(\R^n)}
\]
which can be rewritten as follows
\[
\left( \int_R \| \chi_{S_t(r)} e^{it\varphi_2(D)} \psi_0  \|^2_{L^2_x(\R^n)} dt \right)^\frac12 \ls  r^\frac12 \| \psi_0 \|_{L^2(\R^n)}.
\]
The dual estimate is
\[
\| \int_\R e^{-it\varphi_2(D)} (\chi_{S_t(r)} F(t) ) dt \|_{L^2(\R^n)} \ls  r^\frac12 \| F \|_{L^2(\R^{n+1})}
\]
where $F$ inherits the Fourier localization properties of $\psi$. The usual $TT^*$ argument implies that establishing either of the two is equivalent to proving the following estimate
\[
\left( \int_R \| \chi_{S_t(r)} \int e^{i(t-s)\varphi_2(D)} \chi_{S_s(r)} F(s) ds   \|^2_{L^2_x(\R^n)} dt \right)^\frac12 \ls 
 r \| F \|_{L^2(\R^{n+1})}.
\]
If $|s-t| \les r$, the estimate follows from the isometry property of $e^{i(t-s)\varphi_2(D)}$ on $L^2_x(\R^n)$. 

At larger time scales differences, that is $|s-t| \gg r$, we write the estimate as
\begin{equation} \label{ESre}
\| \int \chi_{S_t(r)} K(t-s,x-y) \chi_{S_s(r)} F(s,y) dy ds \|_{L^2(R^{n+1})} \ls r \| F \|_{L^2(\R^{n+1})},
\end{equation}
where the kernel $K$ is given by
\[
K(x,t) = \int e^{-i(x\cdot \xi+t \varphi_2(\xi))} \eta(\xi) d \xi
\]
with $\eta$ chosen so as to reflect the support properties of $F$, which in turn are derived from those of $\psi$:
$ \eta$ is supported on $D_2$. The gradient of the phase function above $\alpha(\xi)= x\cdot \xi+t \varphi_2(\xi)$ is $\nabla \alpha = x + t \nabla \varphi_2(\xi)$ and it can be easily seen that
$|\nabla \alpha(\xi)| \ges 1$ for $(x,t) \notin \mathcal{CN} (S_2)= \{ \lambda N_2(\zeta): \zeta \in S_2, \lambda \in \R \}$(to get a uniform estimate below, one needs to strengthen $(x,t) \notin$ neighborhood of $\mathcal{CN} (S_2)$). In that case we have the improved estimate
\[
|K(x,t)| \ls_N (1+|x|+|t|)^{-N}
\]
Now, given two points $(x,t),(y,s) \in S(r)$ with $|t-s| \gg r$, by using the transversality property of $N_2(\zeta)$ to $S$, for any $\zeta \in S_2$, it follows that $(x-y,t-s) \notin \mathcal{CN}(S_2)$.
Therefore we can access the bound above to conclude
\[
\| \int \chi_{S_t(r)} K(t-s,x-y) \chi_{S_s(r)} F(s,y) dy \|_{L^2(\R^{n+1})} \ls_N (|t-s|)^{-N} \| F(s) \|_{L^2(\R^{n+1})}.
\]
This bound is effective, since $|t-s| \gg r \ges 1$, therefore  we obtain \eqref{ESre}.

\end{proof}

\section{Wave packets} \label{SWP}
We start this section by giving a heuristic approach to the wave packet construction. Let $\varphi$ be a smooth function (to be thought of as either $\varphi_1$ or $\varphi_2$) on its domain taken to be a neighborhood of $\tilde D$. In light of \bf C2\rm, we work under the hypothesis that $H\varphi$ is not degenerate, where we recall that $H \varphi$ stands for the Hessian of $\varphi$. 

We start from the following expansion which holds true locally
\[
\varphi(\xi)= \varphi(\xi_0) + \nabla \varphi(\xi_0) \cdot (\xi-\xi_0) + \la H \varphi (\xi-\xi_0), \xi - \xi_0\ra + O(|\xi - \xi_0|^3).
\]
The free wave with initial data $f_0$ is given by
\[
f(x,t)=  e^{it\varphi(D)} f_0 =\int e^{i(x\cdot \xi+ t \varphi(\xi))} \hat f_0(\xi) d \xi.
\]
With the above expansions for $\varphi$ we expand the phase
\[
\begin{split}
x \cdot \xi + t\varphi(\xi) & = x \cdot \xi_0 + x \cdot (\xi - \xi_0) + t \varphi(\xi_0) + t \nabla \varphi(\xi_0) \cdot (\xi-\xi_0) \\
& + t \la H \varphi (\xi-\xi_0), \xi - \xi_0\ra + t O(|\xi - \xi_0|^3)
\end{split}
\]
Each component reveals some information about the flow according to its degree. Heuristically this is read as follows:

- $e^{ix \cdot \xi_0}, e^{it \varphi(\xi_0)}$ describe the space, respectively time oscillation of the free wave: spatial frequency
$\xi_0$, temporal frequency $\varphi(\xi_0)$,

- $e^{i(x \cdot (\xi - \xi_0) + t \nabla \varphi(\xi_0) \cdot (\xi-\xi_0))}$ describes the space-time region of concentration of the wave, which is the set
of stationary points of the phase function. This is the region described by the equation $x + t \nabla \varphi(\xi_0)=0$, which in particular identifies the propagation velocity for the waves to be $-\nabla \varphi(\xi_0)$,

- the quadratic or higher order terms describe the additional time oscillation of the wave and decide the shape of the wave packets. 

Since we assume that $H\varphi$ is non-degenerate, let $\xi$ such that $|\la H \varphi (\xi-\xi_0), \xi - \xi_0\ra| \ges |\xi -\xi_0|^2$ (simply choose
$\xi$ such that $\xi-\xi_0$ is an eigenvector corresponding to a non-zero eigenvalue). Then the additional time oscillation becomes effective once $|t| \cdot |\xi-\xi_0|^2 \ges 1$. This suggests that, for a given time interval $[0,T]$, the correct scale for frequency localization is $| \Delta \xi | \les T^{-\frac12}$. To make the process efficient, the wave packets are chosen at the sharp scales obeying the uncertainty principle, therefore the (dual) localization on the physical side should be $|\Delta x| \ls T^\frac12$. 

One area of potential concern is what happens with the higher order terms in the expansion of the phase. By the same token,
a cubic or higher order components, that is $t|\xi-\xi_0|^k, k \geq 3$ terms would require localizations at scale 
$|\Delta \xi| \ls T^{-\frac1k}, |\Delta x| \ls T^{\frac1k}$. But this implies that the localization
dictated by the quadratic phase works well for the higher order terms. A more direct computation is 
$|t| |\xi-\xi_0|^k \leq T T^{-\frac{k}2} = T^{1-\frac{k}2} \ll 1$ since our time scales are taken to be large. 

We now continue with the formalization of the wave packet construction. Let $\mathcal{L}=r^{-1} \Z^n \cap D$ and let $L$ be the lattice $L=c^{-2} r \Z^n$. With $x_T \in L, \xi_T \in \mathcal{L}$ we define the tube $T:=\{ (x,t) \in \R^n \times \R: |x-x_T + \nabla \varphi(\xi_T) t| \leq c^{-2} r \}$ and denote by $\calT$ the set of such tubes. Associated to a tube $T \in \calT$, we define the cut-off $\tilde \chi_T$ on $\R^{n+1}$ by
\[
\tilde \chi_T(x,t)= \tilde \chi_{D(x_T - \nabla \varphi(\xi_T) t,t; c^{-2} r)}(x).
\] 
 
Usually the parameter $c$ is chosen $\approx 1$. In our context working with $c \ll 1$ plays a crucial role in keeping tight bounds 
on various quantities, see for instance \eqref{q00e} below. 
  
The following result describes the wave packet decomposition we use in this paper and it is inspired by a similar construction found in \cite{Tao-BW} in the context of the Wave equation.
\begin{lema} \label{LeWP} Let $Q$ be a cube of radius $R \gg 1$, let $c$ be such that $R^{-\frac14+} \ll c \ll 1$ and let $J \in \N$ be such that $r =2^{-J} R \approx R^\frac12$. Let $f(t)=e^{it\varphi(D)} f(0)$ be a free wave with $margin(f) > 0$. For each $T \in \calT$ there is a free solution 
$f_T$, with $\hat f_T$ supported in a cube of size less than $CR^{-\frac12}$ and obeying $margin(f_T) \geq margin(f)-CR^{-\frac12}$. The map $ f \rightarrow f_T$ is linear and 
\begin{equation} \label{lind}
f= \sum_{T \in \calT} f_{T}.
\end{equation}
If $\mbox{dist}(T,Q) \geq 4 R$ then 
\begin{equation} \label{ld}
\| f_T \|_{L^\infty(Q)} \ls c^{-C} dist(T,Q)^{-N} M(f)^\frac12.
\end{equation}
The following estimates hold true
\begin{equation} \label{qest}
\sum_{T} \sup_{q \in Q_J(Q)} \tilde \chi_T(x_q,t_q)^{-N} \| f_T \|^2_{L^2(q)} \ls c^{-C} r M(f)  
\end{equation}
and
\begin{equation} \label{q00e}
\left( \sum_{q_0} M( \sum_T m_{q_0,T} f_T)  \right)^\frac12
\leq (1+cC) M(f),
\end{equation}
provided that the coefficients $m_{q_0,T} \geq 0$ satisfy
\begin{equation} \label{q0e}
\sum_{q_0} m_{q_0,T}=1, \qquad \forall T \in \calT.
\end{equation}
\end{lema}

Our wave packet decomposition uses the quadratic phase template and it is standard for equations whose characteristic surface is of quadratic type, the standard model being the Sch\"odinger equation. However it is obviously different than the standard wave packet decomposition used for the wave equation. As explained in the beginning of this section, if one seeks a common denominator
for a wave packet theory for surfaces with some curvature, then the natural choice comes from the wave packet construction for surfaces with non-zero Gaussian curvature. 

A wave packet is defined starting with a phase-space decomposition of $\R^n$. This can be achieved in several ways,
most commonly by the composition of two smooth cut-offs, one in frequency and one in space (or in reverse order), which localize at dual scale:
\[
f =  \sum_{x_0} \sum_{\xi_0} \chi_{x_0}(x) \chi_{\xi_0}(D) f
\]
The scales of the two localization have to obey the uncertainty principle, and this is why it is common to chose them dual to each other. An important observation is that $\chi_{x_0}(x) \chi_{\xi_0}(D) f$ cannot have compact support both in phase and in space. Given that the evolution $e^{it\varphi (D)}$ preserves the Fourier support and not the physical one, it is then preferably to use decompositions whose terms have compact Fourier support. Moreover, our statements assume Fourier localizations, and this is 
another reason why the elements in the wave packet decomposition need to have that property as well. If we want 
\[
\mathcal{F} (\chi_{x_0}(x) \chi_{\xi_0}(D) f) = \hat \chi_{x_0} *  \mathcal{F} (\chi_{\xi_0}(D) f),
\]
to have compact (Fourier) support, then $\hat \chi_{x_0}$ needs to have compact support. In addition, 
we want this compact support to not alter too much the support of $\chi_{\xi_0}(\xi)$, when performing the convolution above. 

\begin{proof}[Proof of Lemma \ref{LeWP}] With the above in mind, we start with the  partition
\[
D= \bigcup_{\xi \in \mathcal{L}} A_\xi
\]
where $A_\xi$ consists of the points in $D$ that are closer to $\xi$ than any other elements of $\mathcal{L}$. 
Therefore $A_\xi$ belongs to the $O(r^{-1})$ neighborhood of $\xi$. 

Let $G$ be the set of all translations in $\R^n$ by vectors of size at most $O(r^{-1})$; in particular these translations
 differ from identity by $O(r^{-1})$. Let $d \Omega$
be a smooth compactly supported probability measure on the interior of $G$. 
For each $\Omega \in G$ and $\xi_0 \in \mathcal{L}$, we define the Fourier projectors by
\[
\mathcal{F}(P_{\Omega,\xi_0} g)(\xi) = \chi_{\Omega(A_{\xi_0})} (\xi) \hat g (\xi). 
\]
For fixed $\Omega \in G$, this leads to the decomposition:
\begin{equation} \label{Fdec}
g = \sum_{\xi_0 \in \mathcal{L}} P_{\Omega,\xi_0} g. 
\end{equation}
The terms above have good frequency support and next we proceed with the spatial localization. For each $x_0 \in L$, define
\[
\eta^{x_0}(x) = \eta_0 (\frac{c^2}r (x-x_0))
\]
and notice that, by the Poisson summation formula and properties of $\eta_0$, 
\begin{equation} \label{pois}
\sum_{x_0 \in L} \eta^{x_0}=1.
\end{equation}
Next we define 
\[
f_T(0)= \eta^{x_T}(x) \int P_{\Omega,\xi_T} f(0) d \Omega
\]
and evolve this, at all other times, by the free flow
\[
f_T(t)=e^{it\varphi(D)} f_T(0). 
\]
Without the averaging in $d \Omega$, the above decomposition is a standard wave packet decomposition and it would provide 
all the properties claimed, except \eqref{q0e} with the sharp bounds (\eqref{q0e} would still be true, after replacing $1+cC$
by $C$). We now explain the role of the averaging on $d \Omega$. The localization on the physical side comes in a product fashion 
and then, due to \eqref{pois}, its impact in \eqref{q0e} comes with good bounds. The original localization on the Fourier side \eqref{Fdec} would also have good bounds (or at least it can be redefined to do so), but the final localization on the Fourier side comes
through a convolution process
\begin{equation} \label{conv}
\mathcal{F}(\eta^{x_T}) * \mathcal{F}(P_{\Omega,\xi_T} f(0))
\end{equation}
and this creates the following problem: two packets with neighboring speeds, $|\xi_{T_1}-\xi_{T_2}| \approx r^{-1}$, may contain mass from the same frequency region (due to the convolution process) and this can potentially alter the tight bounds in \eqref{q0e}. A more careful look reveals the following: $\mathcal{F} (\eta^{x_T})$ has Fourier support in the region $|\xi| \ls c^2 r^{-1}$, thus the common region mentioned above does not have volume $\approx r^{-n}$, but instead $\les c^2 r^{-n} \ll r^{-n}$. One would like  to take advantage of by using that $\mathcal{F}(P_{\Omega,\xi_T} f(0))$ has smaller mass on smaller sets, but this is not true for generic $L^2$ functions.  However, the averaging process in $d \Omega$ leads to the desired conclusion and as a consequence the common amount of mass that can be shared by two packets with neighboring speeds can be estimated by factors containing $c$, thus providing the improvement claimed in \eqref{q0e}. 

Now we turn to the proofs of all claims in the Lemma. The linearity of the map $f \rightarrow f_T$ and \eqref{lind} are obvious.  
$P_{\Omega,\xi_T}$ are Fourier projectors, thus do not alter the frequency support; averaging on $d \Omega$ has the same
property. This Fourier support is altered due to the physical localization which is described by the convolution \eqref{conv}. Since 
$\mathcal{F} (\eta^{x_T})$ has Fourier support in the region $|\xi| \ls c^2 r^{-1}$, the margin of the wave $P_{\Omega,\xi_T} f(0)$ is altered by at most $c^2 Cr^{-1} \ll Cr^{-1}$. This implies the margin claim in the Lemma since the flow $e^{it\varphi(D)}$
preserves the Fourier support. 

In order to prove \eqref{ld} and \eqref{qest} we need the following estimate
\begin{equation} \label{comm}
\| |x-x_T + t\nabla \varphi(\xi_T)|^{\alpha} f_T(t)  \|_{L^2(\R^n)} \les_\alpha c^{-2\alpha} r^\alpha \| f(0) \|_{L^2(\R^n)}
\end{equation}
for all $\alpha \in \N$. The estimate is obvious for $\alpha=0$. We establish \eqref{comm} for $\alpha=1$
and note that the argument for general $\alpha$ is similar. We start from the commutator identity
\[
(x-x_T + t\nabla \varphi(D)) e^{it \varphi(D)} = e^{it \varphi(D)}(x-x_T)
\]
which can be checked directly by taking a Fourier transform. Therefore we have
\[
\begin{split}
\| (x-x_T + t\nabla \varphi(D)) e^{it \varphi(D)} f_T(0)\|_{L^2} & = \| e^{it \varphi(D)}(x-x_T) f_T(0)\|_{L^2} \\
& =  \| (x-x_T) f_T(0)\|_{L^2} \\
& \les c^{-2} r \| \int P_{\Omega,\xi_T} f(0) d \Omega \|_{L^2} \\
& \les c^{-2} r \| f(0) \|_{L^2}
\end{split}
\]
where we have used the fast decay properties of $\eta_0$. 

To conclude with \eqref{comm} with $\alpha=1$, we need to replace $t\nabla \varphi(D)$ with 
$ t \nabla \varphi (\xi_T)$ in the above expression. This is done based on the estimate
\[
\begin{split}
\| (t\nabla \varphi(D) - t \nabla \varphi (\xi_T)) e^{it\varphi(D)} f_T(0) \|_{L^2} 
& = \| (t\nabla \varphi(D) - t \nabla \varphi (\xi_T))  f_T(0) \|_{L^2} \\
& = |t| \| (\nabla \varphi(\xi) -  \nabla \varphi (\xi_T))  \hat f_T(0) \|_{L^2} \\
& \les |t| r^{-1} \| D^2 \varphi \|_{L^\infty} \| f_T(0) \|_{L^2} \\
& \les r \| f(0) \|_{L^2}
\end{split}
\]
where we have used: the unitarity of $e^{it\varphi(D)}$ in $L^2_x$, Plancherel and the fact 
 that $|\xi - \xi_T| \leq Cr^{-1}$  for $\xi$ in the support of $\hat f_T$. Combining the two estimates above leads to
 \eqref{comm}. 

Now we prove \eqref{ld}. Let $D(x_D,t_D,2r) \subset 2Q$ be a disk of radius $2r$ and contained in $2Q$. 
From \eqref{comm} it follows that
\[
\| f_T(t_D) \|_{L^2(D)} \les_\alpha c^{-2\alpha} r^\alpha d(T,Q)^{-\alpha} \| f \|_{L^2}
\]
and since $d(T,Q) \geq 4 R \ges r^2$, we obtain
\[
\| f_T(t_D) \|_{L^2(D)} \les c^{-4N}  d(T,Q)^{-N} \| f \|_{L^2}.
\]
Given that $f_T$ is supported at frequency $\approx 1$, it is easy to show that similar estimates hold true for $ \| \partial^\beta f_T(t_D) \|_{L^2(D)}$ for $0 \leq |\beta| \leq \frac{n}2+1$ and this leads to desired $L^\infty$ bounds on a slightly smaller disk.
This implies \eqref{ld}.

From the argument provided for \eqref{comm} we see that for any $\alpha \in \N$
\[
\| |\frac{x_q-x_T + t_q \nabla \varphi(\xi_T)}{r}|^{\alpha} f_T  \|_{L^2(q)} \les_\alpha c^{-2\alpha} r^\frac12 \| f_T^\alpha(0) \|_{L^2}
\]
where $f_T^\alpha(0)=(\frac{|x-x_T|}{r})^\alpha f_T(0)$. The factor of $r^\frac12$ is due to the time integration since $q$ has size $r$ in the time direction.
In order to conclude with \eqref{qest} we need to establish
\[
\sum_T \| f_T^\alpha(0) \|^2_{L^2} \les_\alpha c^{-2\alpha} \| f \|^2_{L^2}. 
\]
This is done in two steps. The summation with respect to $x_T$  follows from
\[
\sum_{x_T \in L} (\frac{|x-x_T|}{r})^\alpha \eta^{x_T}(x) \les c^{-\alpha}
\]
which is a consequence of the fast decay properties of $\eta^0$ (recall also \eqref{pois}). The summation with respect to $\xi_T$ follows from the almost orthogonality of the projectors $P_{\Omega,\xi}$ quantified as follows
\[
\sum_{\xi_T \in \mathcal{L}} \| P_{\Omega, \xi_T} f(0) \|_{L^2} \les \| f(0) \|_{L^2}.
\]
The later estimate remains valid when averaging on $d \Omega$. This finishes the proof of \eqref{qest}.

Finally, we note that, by using the unitarity of $e^{it\varphi(D)}$ on $L^2$, \eqref{q0e} is reduced to the corresponding statement
for $f(0)$ which has nothing to do with the specific flow dictated by $e^{it\varphi(D)}$. But then the statement follows in a completely similar manner to the corresponding one in \cite{Tao-BW}, see Lemma $15.2$, estimate $(63)$ with a proof provided in Appendix 1;
the only adjustment needed is the definition of the set $G$ and correspondingly $d \Omega$, but this does not change at all the structure of the argument. Also the reader may take notice that in the argument of $(63)$ in \cite{Tao-BW} the specific flow (of the wave equation) is absent. 

\end{proof}

\section{Table construction and the induction argument} \label{SI}

This section contains the main argument for the proof of Theorem \ref{mainT}. In Proposition \ref{NLP}
we construct tables on cubes: this is a way of re-organizing the information on one term, say $\phi$, at smaller scales
based on information from the other interacting term $\psi$. This essentially replaces the classical combinatorial argument
used in most of the previous works, and it is inspired by the work on the conic surfaces of Tao in \cite{Tao-BW}. Based on this table
construction, we are then able to prove the inductive bound claimed in Proposition \ref{keyP}. 

\begin{prop} \label{NLP}
Let $Q$ be a cube of size $R \gg 2^{2C_0}$ and let 
$c > 0$ such that $ R^{-\frac14} \ll c \ll 1$.
Let $\phi=e^{it\varphi_1(D)} \phi_0, \psi=e^{it\varphi_2(D)} \psi_0$ be free waves with positive margin
relatively to $\tilde S_1$ respectively $\tilde S_2$.  Then there is a table $\Phi=\Phi_c(\phi,\psi,Q)$ with depth $C_0$
such that the following properties hold true:
\begin{equation} \label{dec}
\phi= \sum_{q \in  \calQ_{C_0}(Q)} \Phi^{(q)}, 
\end{equation}
\begin{equation} \label{MPhi2}
margin(\Phi) \geq margin(\phi)-C R^{-\frac12}.
\end{equation}

\begin{equation} \label{PhiM2}
M(\Phi) \leq (1+cC) M(\phi),
\end{equation}
and for any $q',q'' \in \calQ_{C_0}(Q), q'\ne q''$
\begin{equation} \label{Phia2}
\| \Phi^{(q')} \psi \|_{L^2((1-c)q'')} \ls c^{-C} R^{-\frac{n-1}4}  M^\frac12(\phi) M^\frac12(\psi).
\end{equation}

\end{prop}

\begin{rem} \label{RNLP}
The above result is stated for scalar $\phi,\psi$, but it holds for vector versions as well. Most important 
is that we can construct $\Phi=\Phi_c(\phi, \Psi, Q)$ where $\Psi$ is a vector free wave and all its scalar components 
satisfy similar properties to the $\psi$ above. 
\end{rem}

\begin{proof}  There are several scales involved in this argument. The large scale is the size $R$ of the cube
$Q$. The coarse scale is $2^{-C_0} R \gg R^\frac12$, this being the size of the smaller cubes
in $\calQ_{C_0}(Q)$ and the subject of the claims in the Proposition. Then there is the fine scale $r=2^{-j}R$ chosen such that
 $r \approx R^\frac12$. Notice that $r$ is the proper scale for wave packets corresponding to time scales $R$ and also
 that their scale is $c^{-2} r \ll 2^{-C_0} R$, last one being the scale of cubes in $\calQ_{C_0}(Q)$. 
 
We use Lemma \ref{LeWP} with $J=j$ to construct the wave packet decomposition for $\phi$. For any $q_0 \in \calQ_{C_0}(Q)$ we define
\[
m_{q_0,T}:= \sum_{\xi_2 \in \mathcal{L}} \|  \tilde{\chi}_T \psi_{\xi_2} \|^2_{L^2(q_0)} 
\]
and 
\[
m_T:= \sum_{q_0 \in \calQ_{C_0}(Q)} m_{q_0, T} = \sum_{\xi_2 \in \mathcal{L}} \| \tilde{\chi}_T  \psi_{\xi_2} \|^2_{L^2(Q)}.
\]
Based on this we define
\begin{equation} \label{dPq0}
\Phi^{(q_0)}:= \sum_{T} \frac{m_{q_0,T}}{m_T} \phi_T. 
\end{equation}
One are of concern may be the fact that $m_T=0$ for some tube $T$ and that would create problems in the definition above. 
In this case it follows that $\tilde{\chi}_T  \psi_{\xi_2}=0$ on $Q$, for all $\xi_2$, thus $\tilde{\chi}_T  \psi=0$ on $Q$ which implies that
$\psi=0$ on $Q$. But in this case, any table $\Phi$ satisfies \eqref{Phia2}. Constructing tables satisfying the other properties is a trivial matter,
we can simply replace the degenerate coefficients $\frac{m_{q_0,T}}{m_T}$ in \eqref{dPq0} by $\frac1{2^{(n+1)C_0}}$. 

By combing the definitions above with the decomposition property \eqref{lind}, we obtain 
\[
\phi=\sum_{q_0 \in \calQ_{C_0}(Q)} \Phi^{(q_0)}
\]
thus justifying \eqref{dec}. 

The margin estimate \eqref{MPhi2} follows from the margin estimate on tubes provided by Lemma \ref{LeWP}. 
The coefficients $m_{q_0,T}$ satisfy \eqref{q0e}, thus the estimate \eqref{PhiM2} follows from 
 \eqref{q00e}. 

All that is left to prove is \eqref{Phia2}, which is equivalent to 
\begin{equation} \label{red1}
\sum_{q \in \calQ_j(Q):d(q,q_0) \ges cR} \| \Phi^{(q_0)} \psi \|^2_{L^2(q)} \ls c^{-C} r^{-(n-1)} M(\phi) M(\psi).
\end{equation}
Note that the cubes $q$ are selected at the finer scale dictated the size of cubes in $\calQ_j(Q)$. From the definition of $\Phi^{(q_0)}$ in 
\eqref{dPq0} we can discard the tubes $q$ which do not intersect $4Q$ based on \eqref{ld}, in the sense that their contribution 
to \eqref{red1} will give a better estimate. 

For the tubes intersecting $4Q$, we make another simplification motivated  by \eqref{qest} 
and focus on the tubes which intersect $q$, that is we focus on the following term 
\[
\sum_{q \in \calQ_j(Q):d(q,q_0) \ges cR} \| \sum_{T \cap q \ne \emptyset} \frac{m_{q_0,T}}{m_T} \phi_T \psi \|^2_{L^2(q)}.  
\]
Essentially \eqref{qest}  says that the other tubes have off-diagonal type contribution, that is there are enough gains in the case $T \cap q = \emptyset$ to perform any summation, see the commentaries at end of the proof.

We further expand the above term as follows
\[
= \sum_{q \in \calQ_j(Q):d(q,q_0) \ges cR} \| \sum_{\xi_2 \in \L} \sum_{T_1 \cap q \ne \emptyset} \frac{m_{q_0,T_1}}{m_{T_1}} \phi_{T_1} \psi_{\xi_2} \|^2_{L^2(q)},
\]
where the use of $T_1$ here versus $T$ has no other meaning than streamlining notations. 

We bound the inner summand by
\[
\| \sum_{\xi_2 \in \L} \sum_{T_1 \cap q \ne \emptyset} \frac{m_{q_0,T_1}}{m_{T_1}} \phi_{T_1} \psi_{\xi_2} 
\tilde\chi_q\|^2_{L^2}
\]
where $\tilde \chi_q$ is a smooth approximation of the characteristic function of $q$. More precisely $\tilde\chi_q \equiv 1$
on $q$, $\tilde\chi_q \equiv 0$ on $\R^{n+1} \setminus 2q$ and $|\partial^\alpha_{x,t} \tilde\chi_q | \ls_\alpha r^{-|\alpha|}$ for all 
multi-indexes $\alpha \in \N^{n+1}$. As a consequence, on the Fourier side $\mathcal{F}_{x,t}(\tilde\chi_q)$ is highly concentrated in the region $|(\xi,\tau) | \leq r^{-1}$ and decays fast away from it, that is $|\mathcal{F}(\tilde\chi_q)(\xi,\tau)| \ls_N \la r (\xi,\tau) \ra^{-N}$ for all $N \in \N$. 

Since $\mathcal{F}_{x,t} (\phi_{T_1} \psi_{\xi_2} \tilde\chi_q)= \mathcal{F}_{x,t} (\phi_{T_1} \psi_{\xi_2}) 
* \mathcal{F}_{x,t} (\tilde\chi_q)$, it follows that the multiplication by $\tilde\chi_q$ does not change, morally speaking, the space-time Fourier support of the product $ \phi_{T_1} \psi_{\xi_2}$ by more than $r^{-1}$.

We aim to exploit two types of orthogonality in the interactions $\phi_{T_1} \cdot \psi_{\xi_2}$ between $\phi_{T_1}$ and $\psi_{\xi_2}$: on the spatial frequency side and on the temporal frequency side. The above observation will allow to claim some stability of this orthogonality after multiplication by $\tilde\chi_q$.

The term $ \phi_{T_1} \psi_{\xi_2}$ has spatial frequency $\xi_1+\xi_2$ and time frequency $\varphi_1(\xi_1) + \varphi_2(\xi_2)$, in the sense that its frequency support belongs to the set $\{(\xi,\tau): |(\xi,\tau)-(\xi_1+\xi_2, \varphi_1(\xi_1) + \varphi_2(\xi_2))| \les r^{-1}\}$. 
Therefore, using an almost orthogonality argument based on the decay properties of $\mathcal{F}_{x,t} (\tilde\chi_q)$, the following holds true
\begin{equation} \label{sq}
\| \sum_{\xi_2} \sum_{T_1 \cap q \ne \emptyset} \frac{m_{q_0,T_1}}{m_{T_1}} \phi_{T_1} \psi_{\xi_2} \tilde\chi_q \|_{L^2}^2 \ls \sum_{\xi \in \L} \sum_{\tau \in \L_1}  \| \sum_{(\xi_1, \xi_2) \in A(\xi,\tau)} \sum_{T_1 \cap q \ne \emptyset: \atop \xi_{T_1}=\xi_1} \frac{m_{q_0,T_1}}{m_{T_1}} \phi_{T_1} \psi_{\xi_2} \tilde\chi_q \|_{L^2}^2.
\end{equation}
Here by $(\xi_1, \xi_2) \in A(\xi,\tau)$ we mean that $|\xi_1 + \xi_2 - \xi|, |\varphi_1(\xi_1) + \varphi_2(\xi_2)-\tau| \les r^{-1}$ and 
$(\xi,\tau) \in \L \times \L_1$ where $\L_1=r^{-1}\Z$ (in other words $\L \times \L_1=r^{-1} \Z^{n+1}$). One way to think of the above is that $\xi_2$ is almost uniquely determined by $\xi_1 \in A_1(\xi,\tau)$ via $\xi_2=\xi-\xi_1+ \tilde \xi, \tilde{\xi} \in \L, |\tilde \xi| \les r^{-1}$, where $A_1(\xi,\tau)$ is the set of $\xi_1$ for which there exists a $\xi_2$ such that $(\xi_1,\xi_2) \in A(\xi,\tau)$. 

We now unravel some key observations about the set $A_1(\xi,\tau)$. Note that the set of solutions of the equation $ (\xi_1, \varphi_1(\xi_1)) + (\xi_2,\varphi_2(\xi_2))= \beta$ is the set $S_1 \cap \tau^{\beta}(- S_2)= C_1(\beta)$.
Let $S:=\mathcal{CN}(C_1(\beta)) =\{  \alpha N_1(\zeta): \zeta \in C_1(\beta), \alpha \in \R \}$ where $\beta \in \L \times \L_1$ such that
$|\beta-(\xi_1+\xi_2, \varphi_1(\xi_1) + \varphi_2(\xi_2))| \les r^{-1}$.
We can conclude that the "thickened" surface
\[
\tilde S:=\{ T_1: \xi_1 \in A_1(\xi,\tau), T_1 \cap q \ne \emptyset, T_1 \cap q_0 \ne \emptyset  \}
\]
has the property that $\tilde S \cap q_0$ is a subset of the intersection of $q_0 \cap ((x_q,t_q)+S(c^{-2}r))$ where
we recall that $S(c^{-2}r)$ is the neighborhood of size $c^{-2}r$ to $S$. 

Now, for fixed $(\xi,\tau)$ we write
\[
 \sum_{\xi_1 \in A_1(\xi,\tau)} \sum_{T_1 \cap q \ne \emptyset: \atop \xi_{T_1}=\xi_1} \frac{m_{q_0,T_1}}{m_{T_1}} \phi_{T_1} \psi_{\xi_2} = \sum_{T_1 \in \T(A_1(\xi,\tau))} \frac{m_{q_0,T_1}}{m_{T_1}} \phi_{T_1} \psi_{\xi_2}
\]
where  $\calT(A_1(\xi,\tau)) = \{ T_1 \in \calT: T_1 \cap q \ne \emptyset, \xi_{T_1}=\xi_1, \xi_1 \in A_1(\xi,\tau) \}$ and $\xi_2$ is explicitly determined by $T_1$ through $\xi_1$ as described above. Using the above and the obvious inequality $\frac{m_{q_0,T_1}}{m_{T_1}} \leq \frac{m^\frac12_{q_0,T_1}}{m^\frac12_{T_1}}$, we obtain:
\[
\begin{split}
& \|  \sum_{\xi_1 \in A_1(\xi,\tau)} \sum_{T_1 \cap q \ne \emptyset: \atop \xi_{T_1}=\xi_1}  \frac{m_{q_0,T_1}}{m_{T_1}} \phi_{T_1} \psi_{\xi_2} \tilde\chi_q \|_{L^2}\\
 \ls & 
\left( \sum_{T_1 \in \calT(A_1(\xi,\tau))} \frac{\|\phi_{T_1} \psi_{\xi_2} \tilde\chi_q \|^2_{L^2}}{m_{T_1} \tilde \chi_{T_1}(x_q,t_q)} \right)^\frac12
 \left(  \sum_{T_1 \in \calT(A_1(\xi,\tau))}  m_{q_0,T_1} \tilde \chi_{T_1}(x_q,t_q)  \right)^\frac12.
\end{split}
\]
Next we claim the following estimate
\begin{equation} \label{keyw}
\begin{split}
\sum_{T_1 \in \calT(A_1(\xi,\tau))} m_{q_0,T_1} \tilde \chi_{T_1}(x_q,t_q) & \ls \sum_{\xi_2 \in \L} \| \chi \psi_{\xi_2} \|_{L^2}^2 \\
& \ls  c^{-C} r \sum_{\xi_2 \in \L} M(\psi_{\xi_2}) \ls c^{-C}  r M(\psi). 
\end{split}
\end{equation}
Using the definition of $m_{q_0,T_1}$ we identify  the function 
\[
\chi= ( \sum_{T_1 \in \calT(A_1(\xi,\tau))} \tilde \chi(x_q,t_q) \tilde \chi_{T_1} ) \chi_{q_0}
\]
which makes the first inequality in \eqref{keyw} true. Then we note that $\chi$ has the following decay property:
\[
\chi (x,t) \ls c^{-4} \left( 1+ \frac{d((x,t), S)}{c^{-2}r}\right)^{-N}.
\]
This is a consequence of the fact that the tubes $T_1$ passing thorough $q$ separate inside $q_0$ as a consequence of \eqref{C3} (see \bf C2\rm) and the separation between $q$ and $q_0$, that is $d(q,q_0) \ges cR$. Quantitatively speaking, given a point in $q_0$ close to $S$, there are $\ls c^{-4}$ tubes $T_1$ passing through the point and $q$.

Based on the decay estimate for $\chi$, we can use \eqref{ES} for each $\psi_{\xi_2}$ to justify the second inequality in \eqref{keyw}. The last inequality  \eqref{keyw} is obvious. 

Next we claim the following estimate:
\begin{equation} \label{estau1}
\sum_q \sum_{\xi \in \L} \sum_{\tau \in \L_1} \sum_{T_1 \in  \calT(A_1(\xi,\tau))} \frac{\|\phi_{T_1} \psi_{\xi_2} \tilde \chi_q \|^2_{L^2}}{m_{T_1} \tilde \chi_{T_1}(x_q,t_q)} \ls r^{-n} M(\phi). 
\end{equation}
Notice that this estimate brings back the summation with respect to $(\xi,\tau)$ from \eqref{sq} together with the original summation with respect to $q$. Combing \eqref{estau1} with \eqref{keyw} gives \eqref{red1} and this concludes the proofs of all claims of the Proposition. 

For the reminder of this proof, we establish \eqref{estau1}. Taking into account the frequency localization of $\phi_{T_1} \psi_{\xi_2}$
and the fast decay properties of $\mathcal{F}_{x,t} (\tilde\chi_q)$, we obtain
\[
\|\phi_{T_1} \psi_{\xi_2} \tilde \chi_q \|^2_{L^2} \ls r^{-(n+1)} \|\phi_{T_1} \psi_{\xi_2} \tilde \chi_q \|^2_{L^1}
\]
Therefore it suffices to show that
\[
\sum_q \sum_\xi \sum_\tau \sum_{T_1 \in  \calT(A_1(\xi,\tau))} \frac{\|\phi_{T_1} \tilde \chi_q \|^2_{L^2}
\| \psi_{\xi_2} \tilde \chi_q \|^2_{L^2}}{m_{T_1} \tilde \chi_{T_1}(x_q,t_q)} \ls r M(\phi).
\]
The summation with respect to $(\xi,\tau)$ brings back all possible frequency interactions, 
hence the above is equivalent to proving
\[
\sum_q  \sum_{T_1 \cap q \ne \emptyset} \sum_{\xi_2} \frac{\|\phi_{T_1} \tilde \chi_q \|^2_{L^2}
\| \psi_{\xi_2} \tilde \chi_q \|^2_{L^2}}{m_{T_1} \tilde \chi_{T_1}(x_q,t_q)} \ls r M(\phi).
\]
Note that in the above estimate the frequency of $T_1$, $\xi_{T_1}=\xi_1$ is decoupled from $\xi_2$, and that the summation
over $T_1$ is essentially a summation over $\xi_1$.

By rearranging the sum, it suffices to show
\[
\sum_{T_1} \sum_{q \cap T_1 \ne \emptyset}  \sum_{\xi_2}  \frac{\|\phi_{T_1} \tilde \chi_q \|^2_{L^2} 
\| \psi_{\xi_2} \tilde \chi_q \|^2_{L^2}}{m_{T_1} \tilde \chi_{T_1}(x_q,t_q)} \ls r M(\phi).
\]
The inner sum is estimated as follows
\[
\sum_{q \cap T_1 \ne \emptyset}  \sum_{\xi_2}  \frac{ \| \psi_{\xi_2} \tilde \chi_q \|^2_{L^2}}{m_{T_1} \tilde \chi_{T_1}(x_q,t_q)}
\ls \sum_{\xi_2}  \frac{ \| \psi_{\xi_2} \tilde \chi_{T_1} \|^2_{L^2}}{m_{T_1}} \ls 1.
\]
What is left is to show is that
\[
\sum_{T_1} \sup_{q} \|\phi_{T_1} \tilde \chi_q \|^2_{L^2}  \ls r \sum_{T_1} M(\phi_{T_1}) \ls r M(\phi),
\]
which is obvious given the size of $q$ in the temporal direction. 

One may notice the similarity of the last inequality and the stronger \eqref{qest}. We recall that we provided a simplified version
of the proof where, at some point in the proof, we assumed that $T_1 \cap q \ne \emptyset$. When considering the general case,
one needs to use the stronger \eqref{qest} which brings additional and enough (by taking $N$ large) decay when $T_1 \cap q = \emptyset$.  The details are left to the reader. 

\end{proof}

\begin{proof}[Proof of Proposition \ref{keyP}] Let $\phi=e^{it\varphi_1(D)} \phi_0,\psi=e^{it\varphi_2(D)} \psi_0$ be
free waves satisfying the margin requirements \eqref{mrpg}. Let $Q_R$ be an arbitrary cube of radius $R$. From Lemma \ref{AVL}
it follows that there is a cube $Q \subset 4Q_R$ of size $2R$ such that
\begin{equation} \label{aux21}
\| \phi \cdot \psi \|_{L^p(Q_R)} \leq (1+cC) \| \phi \cdot \psi \|_{L^p(I^{c,j}(Q))}.
\end{equation}
Using the result of Proposition \ref{NLP} we build the table $\Phi=\Phi_c(\phi,\psi,Q)$ on $\phi$ with depth $C_0$ and estimate as follows
\begin{equation*}
\begin{split}
\| \phi \cdot \psi \|_{L^p(I^{c,C_0}(Q))} & \leq \sum_{q_0,q_0' \in \calQ_{C_0}(Q)} \| \Phi^{(q_0)} \psi \|_{L^p((1-c)q_0')} \\ 
&  \leq \sum_{q_0 \in \calQ_{C_0}(Q)} \left( \| \Phi^{(q_0)} \psi \|_{L^p((1-c)q_0)} + \sum_{q_0' \in \calQ_{C_0}(Q) \setminus \{q_0\}} \| \Phi^{(q_0)} \psi \|_{L^p((1-c)q_0')} \right). 
\end{split}
\end{equation*}
Then we construct a table on $\psi$, $\Psi=\Phi_c(\Phi,\psi,Q)$  with depth $C_0$ and estimate in a similar manner (see Remark \ref{RNLP} after Proposition \ref{NLP}) to obtain
\begin{equation*}
\begin{split}
 \| \Phi^{(q_0)} \psi \|_{L^p((1-c)q_0)}  & \leq \sum_{q_0' \in \calQ_{C_0}(Q)} \| \Phi^{(q_0)} \Psi^{(q_0')} \|_{L^p((1-c)q_0)} \\ 
&  \leq \| \Phi^{(q_0)} \Psi^{(q_0)} \|_{L^p((1-c)q_0)} +  \sum_{q_0' \in \calQ_{C_0}(Q) \setminus \{ q_0 \}} \| \Phi^{(q_0)} \Psi^{(q_0')} \|_{L^p((1-c)q_0)}. 
\end{split}
\end{equation*}
Based on the property \eqref{Phia2} of tables we conclude that, for each $q_0, q_0'$ with $q_0 \ne q_0'$ the following hold true
\[
\| \Phi^{(q_0)} \psi \|_{L^2((1-c)q_0')} + \| \Phi^{(q_0)} \Psi^{(q_0')} \|_{L^2((1-c)q_0)} \leq C c^{-C} R^{-\frac{n-1}4} M(\phi)^\frac12 M(\psi)^\frac12. 
\]
Given the size of the cubes, we easily obtain the $L^1$ estimate
\[
\| \Phi^{(q_0)} \psi \|_{L^1((1-c)q_0')} + \| \Phi^{(q_0)} \Psi^{(q_0')} \|_{L^1((1-c)q_0)} \leq C R M(\phi)^\frac12 M(\psi)^\frac12. 
\]
By interpolation, we obtain the $L^p$ bounds
\[
\| \Phi^{(q_0)} \psi \|_{L^p((1-c)q_0')} + \| \Phi^{(q_0)} \Psi^{(q_0')} \|_{L^p((1-c)q_0)} \leq C c^{-C} R^{\frac{n+3}2(\frac1p-\frac{n+1}{n+3})}  M(\phi)^\frac12 M(\psi)^\frac12, 
\]
which holds true for any $q_0 \ne q_0'$. We plug this in the above estimates to conclude with
\[
\| \phi \cdot \psi \|_{L^p(I^{c,C_0}(Q))}  \leq \sum_{q_0 \in \calQ_{C_0}(Q)} \| \Phi^{(q_0)} \Psi^{(q_0)} \|_{L^p((1-c)q_0)} +
C c^{-C} R^{\frac{n+3}2(\frac1p-\frac{n+1}{n+3})}  M(\phi)^\frac12 M(\psi)^\frac12.
\]
Next we recall that $q_0$ has size $\frac{4R}{2^{C_0}} \leq \frac{R}2$ (which in fact may be seen as setting the threshold needed for $C_0$).
Using this we conclude that
\[
\begin{split}
\| \phi \cdot \psi \|_{L^p(I^{c,C_0}(Q))} &  \leq  \sum_{q_0 \in \calQ_{C_0}(Q)} \bar A_p(\frac{R}2) M(\Phi^{(q_0)})^\frac12 M( \Psi^{(q_0)} )^\frac12 \\
& +  C c^{-C} R^{\frac{n+3}2(\frac1p-\frac{n+1}{n+3})}  M(\phi)^\frac12 M(\psi)^\frac12 \\
& \leq \bar A_p(\frac{R}2) \left( \sum_{q_0 \in \calQ_{C_0}(Q)} M(\Phi^{(q_0)}) \right)^\frac12 
\left( \sum_{q_0 \in \calQ_{C_0}(Q)} M(\Psi^{(q_0)}) \right)^\frac12 \\
& + C c^{-C} R^{\frac{n+3}2(\frac1p-\frac{n+1}{n+3})}  M(\phi)^\frac12 M(\psi)^\frac12 \\
& \leq \bar A_p(\frac{R}2) M(\Phi)^\frac12 M(\Psi)^\frac12 
+ C c^{-C} R^{\frac{n+3}2(\frac1p-\frac{n+1}{n+3})}  M(\phi)^\frac12 M(\psi)^\frac12 \\
& \leq \left( (1+cC) \bar  A_p(\frac{R}2) + C c^{-C} R^{\frac{n+3}2(\frac1p-\frac{n+1}{n+3})} \right) M(\phi)^\frac12 M(\psi)^\frac12. 
\end{split}
\] 
where we have used \eqref{PhiM2} in the last line. In using the induction-type bound on $\Phi^{(q_0)} \Psi^{(q_0)}$ we are
using the margin bounds on $\Phi, \Psi$ from \eqref{MPhi2} to conclude with \eqref{mrpg}; this is easily seen to be the case
provided $R$ is large enough to satisfy $CR^{-\frac12} \leq R^{-\frac14}$. 

Recalling \eqref{aux21}, we obtain that for any cube $Q_R$ of size $R$ the following holds true
\[
\| \phi \psi \|_{L^p(Q_R)} \leq (1+cC) \left( (1+cC) \bar A_p(\frac{R}2) + C c^{-C} R^{\frac{n+3}2(\frac1p-\frac{n+1}{n+3})} \right) M(\phi)^\frac12 M(\psi)^\frac12. 
\]
As a consequence we obtain 
\[
A_p(R) \leq (1+cC) \bar A_p(\frac{R}2) + C c^{-C} R^{\frac{n+3}2(\frac1p-\frac{n+1}{n+3})}.
\]
after redefining $C$, and this is precisely the statement in \eqref{ApR}. 
\end{proof}

\subsection*{Acknowledgement}
Part of this work was supported by a grant from the Simons Foundation ($\# 359929$, Ioan Bejenaru). The author thanks 
Ciprian Demeter, Betsy Stovall and Sanghyuk Lee for helpful suggestions about the literature in the field. Ciprian Demeter pointed out 
that we erroneously stated the main result in \cite{BeCaTa} in the original draft of this paper.

\bibliographystyle{amsplain} \bibliography{HA-refs}

\end{document}